\documentclass[12pt]{article}
\usepackage{geometry}
\usepackage[sort, round]{natbib}
\usepackage{times}
\usepackage{preamble}

\makeatletter
\renewcommand{\section}{\@startsection{section}{1}{\z@}%
   {-0.24in \@plus -1ex \@minus -.2ex}%
   {0.10in \@plus.2ex}%
   {\normalfont\rmfamily\bfseries\large\raggedright}%
}
\renewcommand\subsection{\@startsection{subsection}{2}{\z@}%
   {-0.20in \@plus -1ex \@minus -.2ex}%
   {0.08in \@plus .2ex}%
   {\normalfont\rmfamily\bfseries\normalsize\raggedright}%
}
\renewcommand\subsubsection{\@startsection{subsubsection}{3}{\z@}%
   {-0.18in \@plus -1ex \@minus -.2ex}%
   {0.08in \@plus .2ex}%
   {\normalfont\normalsize\rmfamily\mdseries\itshape\raggedright}%
}
\makeatother

\newcommand{\yang}[1]{{\color{purple} (Yang: #1)}}
\newcommand{\guannan}[1]{{\color{ForestGreen} (Guannan: #1)}}
\newcommand{\adam}[1]{{\color{DarkOrange} (Adam: #1)}}

\usepackage{authblk}

\title{\Large\bfseries On the Sample Complexity of Stabilizing LTI Systems \\ on a Single Trajectory}
\date{}
\newcommand{\algname}{LTS\tsub{0}}

\author[1]{Yang Hu}
\author[2]{Adam Wierman}
\author[3]{Guannan Qu}
\affil[1]{Tsinghua University, Beijing, China}
\affil[2]{California Institute of Technology, Pasadena, CA, USA}
\affil[3]{Carnegie Mellon University, Pittsburgh, PA, USA}

\begin{document}
\setlength{\parskip}{0pt plus 1pt}                     
\setlength{\abovedisplayskip}{5pt plus 2pt minus 2pt}  
\setlength{\belowdisplayskip}{5pt plus 2pt minus 2pt}  

\renewcommand{\labelenumi}{(\theenumi)}

\maketitle

\begin{abstract}%
  Stabilizing an unknown dynamical system is one of the central problems in control theory.
  In this paper, we study the sample complexity of the learn-to-stabilize problem in Linear Time-Invariant (LTI) systems on a single trajectory.
  Current state-of-the-art approaches require a sample complexity linear in $n$, the state dimension, which incurs a state norm that blows up exponentially in $n$.
  We propose a novel algorithm based on spectral decomposition that only needs to learn ``a small part'' of the dynamical matrix acting on its unstable subspace.
  We show that, under proper assumptions, our algorithm stabilizes an LTI system on a single trajectory with $\tilde{O}(k)$ samples, where $k$ is the instability index of the system. This represents the first sub-linear sample complexity result for the stabilization of LTI systems under the regime when $k = o(n)$.
\end{abstract}

  \section{Introduction}

Linear Time-Invariant (LTI) systems, namely $x_{t+1} = Ax_t + Bu_t,$
where $x_t \in \R^n$ is the state and $u_t \in \R^m$ is the control input, are one of the most fundamental dynamical systems in control theory, and have wide applications across engineering, economics and societal domains.
For systems with known dynamical matrices $(A,B)$, there is a well-developed theory for designing feedback controllers with guaranteed stability, robustness, and performance \citep{doyle2013feedback,dullerud2013course}.
However, these tools cannot be directly applied when $(A,B)$ is unknown.

Driven by the success of machine learning \citep{levine2015endtoend,duan2016benchmarking}, there has been significant interest in learning-based (adaptive) control, where the learner does not know the underlying system dynamics and learns to control the system in an online manner, usually with the goal of achieving low regret \citep{fazel2018global,bu2019lqr,li2019distributed,bradtke1994adaptive,tu2017least,krauth2019finite,zhou1996robust,dean2019sample,abbasi2011regret}. 

Despite the progress, an important limitation in this line of work is a common assumption that the learner has a priori access to a known \emph{stabilizing} controller. This assumption simplifies the learning task, since it ensures a bounded state trajectory in the learning stage, and thus enables the learner to learn with an preferably low regret.
However, assuming a known stabilizing controller is by no means practical, as \emph{stabilization} itself is a nontrivial task, and is considered equally important as any performance guarantee like regret or the cost in Linear Quadratic Regulator (LQR).

To overcome this limitation, in this paper we consider the \emph{learn-to-stabilize} problem, i.e., learning to stabilize an unknown dynamical system without prior knowledge of any stabilizing controller.
Understanding the learn-to-stabilize problem is  of great importance to the learning-based control literature, as it serves as a precursor to any learning-based control algorithms that assume knowledge of a stabilizing controller. 

The learn-to-stabilize problem has attracted extensive attention recently. For example, \citet{lale2020explore} and \citet{Hazan2021LearnLTI} adopt a model-based approach that first excites the open-loop system to learn dynamical matrices $(A, B)$, and then designs a stabilizing controller, with a sample complexity scaling linearly in $n$, the state dimension.
However, a linearly-scaling sample complexity is far from satisfactory, since the state trajectory still blows up exponentially when the open-loop system is unstable, incurring a $2^{\tilde{\varTheta}(n)}$ state norm, and hence a $2^{\tilde{\varTheta}(n)}$ regret (in LQR settings, for example).
Another recent work by \citet{Simchowitz2021StabilizePG} proposes a policy-gradient-based discount annealing method that solves a series of discounted LQR problems with increasing discount factors, and shows that the control policy converges to a near-optimal policy. However, this model-free approach only guarantees a $\poly(n)$ sample complexity in the worst case.
In fact, to the best of our knowledge, state-of-the-art learn-to-stabilize algorithms with theoretical guarantees always incur state norms exponential in $n$, which is prohibitively large for high-dimensional systems.

The exponential scaling in $n$ may seem inevitable since, as taking the information-theoretic perspective, a complete recovery of $A$ should take $\varTheta(n)$ samples since $A$ itself involves $n^2$ parameters.
However, our work is motivated by the observation that it is not always necessary to learn the whole matrix $A$ to stabilize an LTI system. For example, if the system is open-loop stable, we do not need to learn anything to stabilize it.
For general LTI systems, it is still intuitive that open-loop \emph{stable ``modes''} exist and need not be learned for the learn-to-stabilize problem. So, we focus on learning a controller that stabilizes the \emph{unstable ``modes''}, making it possible to learn a stabilizing controller without exponentially exploding state norms. The central question of this paper is:

\vspace*{-8pt}
\begin{center}
  \emph{Can we learn to stabilize an LTI system on a single trajectory}\\
  \emph{without incurring a state norm exponentially large in $n$?}
\end{center}
\vspace*{-8pt}

\textbf{Contribution.} In this paper, we answer the above question by designing an algorithm that stabilizes an LTI system with only $\tilde{O}(k)$ state samples along a single trajectory, where $k$ is the \textit{instability index} of the open-loop system and is defined as the number of unstable ``modes'' (i.e., eigenvalues with moduli larger than $1$) of matrix $A$. Our result is significant in the sense that $k$ can be considerably smaller than $n$ for practical systems and, in such cases, our algorithm stabilizes the system using asymptotically fewer samples than prior work; specifically, it only incurs a state norm (and regret) in the order of $2^{\tilde{O}(k)}$, which is much smaller than $2^{O(n)}$ of prior art when $k \ll n$. 

To formalize the concept of unstable ``modes'' for the presentation of our algorithm and analysis, we formulate a novel framework based on the spectral decomposition of dynamical matrix $A$.
More specifically, we focus on the \textit{unstable subspace} $E_{\mathrm{u}}$ spanned by the eigenvectors corresponding to unstable eigenvalues, and consider the system dynamics ``restricted'' to it --- states are orthogonally projected onto $E_{\mathrm{u}}$, and we only have to learn the effective part of $A$ within subspace $E_{\mathrm{u}}$, which takes only $O(k)$ samples. The formulation is explained in detail in Section \ref{sec:algorithm-prelim} and Appendix \ref{sec:appendix-decomp-principle-angle}.
We comment that this idea of decomposition is in stark contrast to prior work, which in one way or another seeks to learn the entire $A$ (or other similar quantities).

\textbf{Related work.} Our work contributes to and builds upon related works described below.

\textit{Learning for control assuming known stabilizing controllers.} There has been a large literature on learning-based control with known stabilizing controllers. For example, one line of research utilizes model-free policy optimization approaches to learn the optimal controller for LTI systems \citep{rautert1997computational,maartensson2009gradient,fazel2018global,malik2018derivative,bu2019lqr,mohammadi2019convergence,li2019distributed,gravell2019learning,yang2019global,zhang2019policy,zhang2020policy,furieri2020learning,jansch2020convergence,jansch2020policy,fatkhullin2020optimizing,tang2021analysis,cassel2021online}. All of these works require a known stabilizing controller as an initializer for the policy search method.
Another line of research uses model-based methods, i.e., learning dynamical matrices $(A,B)$ first before designing a controller, which also require a known stabilizing controller (e.g.,  \citet{faradonbeh2017finite,ouyang2017learning,dean2018regret,cohen2019learning,mania2019certainty,simchowitz2020naive,simchowitz2020improper,zheng2021sample,plevrakis2020geometric}). Compared to these works, we focus on the learn-to-stabilize problem without knowledge of an initial stabilizing controller, which can serve as a precursor to existing learning for control works that require a known stabilizing controller. 

\textit{Learning to stabilize on a single trajectory.} Stabilizing linear systems over \emph{infinite} horizons with asymptotic convergence guarantees is a classical problem that has been studied extensively in a wide range of papers such as \citet{Lai1986StabInfinite, Chen1989StabInfinite, Lai1991StabInfinite}. On the other hand, the problem of system stabilization over \emph{finite} horizons remains partially open and has not seen significant progresses.
Algorithms incurring a $2^{O(n)} O(\sqrt{T})$ regret have been proposed in settings that rely on relatively strong assumptions of controllability and strictly stable transition matrices \citep{abbasi2011regret, Ibrahimi2013Stabilizing}, which has recently been improved to $2^{\tilde{O}(n)} + \tilde{O}(\poly(n) \sqrt{T})$ \citep{lale2020explore, Hazan2021LearnLTI}.
Another model-based approach that merely assumes stabilizability is introduced in \citet{Faradonbeh2019Stabilizing}, though it does not provide guarantees on regret or sample complexity.
A more recent model-free approach based on policy gradient \citep{Simchowitz2021StabilizePG} provides a novel perspective into this problem, yet it can only guarantee a sample complexity that is polynomial in $n$. Compared to these previous works, our approach requires only $\tilde{O}(k)$ samples, and thus incurs a sub-exponential state norm.

\textit{Learning to stabilize on multiple trajectories.} There are also works \citep{dean2019sample,zheng2020non} that do not assume known stabilizing controllers and learn the full dynamics before designing an optimal stabilizing  controller. While requiring $\tilde{\varTheta}(n)$ samples which is larger than $\tilde{O}(k)$ of our work, those approaches do not have the exponentially large state norm issue as they allow \emph{multiple trajectories}; i.e., the state can be ``reset'' to $0$ so that it won't get too large. In contrast, we focus on the more challenging single-trajectory scenario where the state cannot be reset.

\textit{System Identification.} Our work is also related to the system identification literature, which focuses on learning the system parameters of dynamical systems, with early works like \citet{ljung1999system} focusing on asymptotic guarantees, and more recent works such as \citet{simchowitz2018learning,oymak2019non,sarkar2019finite,fattahi2021learning,wang2021large,xing2021identification} focusing on finite-time guarantees. 
Our approach also identifies the system (partially) before constructing a stabilizing controller, but we only identify a part of $A$ rather than the entire $A$.

  \section{Problem Formulation}

We consider a noiseless LTI system $x_{t+1} = Ax_t + Bu_t,$ where $x_t \in \R^n$ and $u_t \in \R^m$ are the \textit{state} and \textit{control input} at time step $t$, respectively. The dynamical matrices $A \in \R^{n \times n}$ and $B \in \R^{n \times m}$ are unknown to the learner. The learner is allowed to learn about the system by interacting with it on a \emph{single trajectory} --- the initial state is sampled uniformly randomly from the unit hyper-sphere surface in $\R^n$, and then, at each time step $t$, the learner is allowed to observe $x_t$ and freely determine $u_t$. The goal of the learner is to learn a stabilizing controller, which is defined as follows.

\begin{definition}[stabilizing controllers]
  Control rule $u_t = f_t(x_t,x_{t-1},\cdots,x_0)$ is called a \textbf{stabilizing controller} if and only if the closed-loop system $x_{t+1} = Ax_t + Bu_t$ is asymptotically stable; i.e., for any $x_0 \in \R^n$, $\lim_{t \to \infty} \snorm{x_t} = 0$ is guaranteed in the closed-loop system.
\end{definition}

To achieve this goal, a simple strategy adopted in prior work \citep{abbasi2011regret, Faradonbeh2019Stabilizing,Hazan2021LearnLTI} is to let the system run open-loop and learn $(A,B)$ (e.g., via least squares), and then design a stabilizing controller based on the learned dynamical matrices. However, as has been discussed in the introduction, such simple strategy inevitably induces an exponentially large stage norm that is unacceptable, and a possible remedy for this is to learn ``a small part'' of $(A,B)$ that is crucial for stabilization. Driven by such intuition, the central problem of this paper is to characterize what is the ``small part'' and design an algorithm to learn it.

Note that, although it is a common practice to include an additive noise term $w_t$ in the LTI dynamics, the introduction of stochasticity does not provide additional insights into our decomposition-based algorithm, but rather, merely adds to the technical complexity of the analysis. Therefore, here we omit the noise in theoretical results for the clarity of exposition, and will show by numerical experiments that our algorithm can also handle noises (see Appendix \ref{sec:appendix-experiment}).

\textbf{Notation.} For $z \in \Cmplx$, $|z|$ is the modulus of $z$. For a matrix $A \in \R^{p \times q}$, $A^{\top}$ denotes the transpose of $A$; $\snorm{A}$ is the induced 2-norm of $A$ (equal to its largest singular value), and $\sigma_{\min}(A)$ is the smallest singular value of $A$; when $A$ is square, $\rho(A)$ denotes the spectral radius (i.e., largest norm of eigenvalue) of $A$. The space spanned by $\sbraced{v_1,\cdots,v_p}$ is denoted by $\Span(v_1,\cdots,v_p)$, and the column space of $A$ is denoted by $\col(A)$. For two subspaces $U, V$ of $\R^n$, $U^{\perp}$ is the orthogonal complement of $U$, and $U \oplus V$ is the direct sum of $U$ and $V$. The zero matrix and identity matrix are denoted by $O$, $I$, respectively.

  \section{Learning to Stabilize from Zero (\algname)}\label{sec:algorithm}

The core of this paper is a novel algorithm, Learning to Stabilize from Zero (\algname), that utilizes a decomposition of the state space based on a characterization of the notion of unstable ``modes''. The decomposition and other preliminaries for the algorithm are first introduced in Section~\ref{sec:algorithm-prelim}, and then we proceed to describe \algname\ in Section~\ref{subsec:algorithm}. 

\subsection{Algorithm Preliminaries}\label{sec:algorithm-prelim}
We first introduce the decomposition of the state space in Section~\ref{subsec:decomposition}, which formally defines the ``small part'' of $A$ mentioned in the introduction. Then, we will introduce $\tau$-hop control in Section~\ref{subsec:tau-hop}, so that we can construct a stabilizing controller based only on the ``small part'' of $A$ (as opposed to the entire $A$). Together, these two ideas form the basis of \algname.

\subsubsection{Decomposition of the State Space}\label{subsec:decomposition}
Consider the open-loop system $x_{t+1} = Ax_t$. Suppose $A$ is diagonalizable, and let $\sbraced{\lambda_1, \cdots, \lambda_n}$ denote the spectrum of $A$, where
$$|\lambda_1| > |\lambda_2| \geq \cdots \geq |\lambda_k| > 1 > |\lambda_{k+1}| \geq \cdots \geq |\lambda_n|.$$
Now we define the \textit{eigenspaces} associated to these eigenvalues: for a real eigenvalue $\lambda_i \in \R$ corresponding to eigenvector $v_i \in \R^n$, associate with it a 1-dimensional space $E_i = \Span(v_i)$; for a complex eigenvalue $\lambda_i \in \Cmplx \setminus \R$ corresponding to eigenvector $v_i \in \Cmplx^n$, there must exist some $j$ such that $\lambda_j = \bar{\lambda}_i$ (corresponding to eigenvector $v_j = \bar{v}_i$), and associate with them a 2-dimensional space $E_{i} = E_{j} = \Span((v_i+\bar{v}_i), \mathrm{i}(v_i-\bar{v}_i))$; for any eigenvalue that appears with multiplicity, the eigenspaces (i.e., eigenvectors $v_i$) are selected to be linearly independent. Further, define the \textit{unstable subspace} $E_{\mathrm{u}} := \bigoplus_{i \leq k} E_i$ and \textit{stable subspace} $E_{\mathrm{s}} := \bigoplus_{i > k} E_i$.

As discussed earlier, we only need to learn ``a small effective part'' of $A$ associated with the unstable ``modes'', or the unstable eigenvectors of $A$. 
For this purpose, in the following we formally define a decomposition based on the orthogonal projection onto the unstable subspace $E_{\mathrm{u}}$. This decomposition forms the foundation of our algorithm.

\textbf{The $\bm{E_{\mathrm{u}} \oplus E_{\mathrm{u}}^{\perp}}$-decomposition.} Suppose the unstable subspace $E_{\mathrm{u}}$ and its orthogonal complement $E_{\mathrm{u}}^{\perp}$ are given by \textit{orthonormal} bases $P_1 \in \R^{n \times k}$ and $P_2 \in \R^{n \times (n-k)}$, respectively, namely
$$E_{\mathrm{u}} = \col(P_1),~
  E_{\mathrm{u}}^{\perp} = \col(P_2).$$
Let $P = [P_1 ~ P_2]$, which is also orthonormal and thus $P^{-1} = P^{\top} = [P_1 ~ P_2]^{\top}$. For convenience, let $\varPi_1 := P_1 P_1^{\top}$ and $\varPi_2 = P_2 P_2^{\top}$ be the \textit{orthogonal} projectors onto $E_{\mathrm{u}}$ and $E_{\mathrm{u}}^{\perp}$, respectively. With the state space decomposition, we proceed to decompose matrix $A$. Note that $E_{\mathrm{u}}$ is an invariant subspace with regard to $A$ (but $E_{\mathrm{u}}^{\perp}$ not necessarily is), there exists $M_1 \in \R^{k \times k}$, $\varDelta \in \R^{k \times (n-k)}$ and $M_2 \in \R^{(n-k) \times (n-k)}$, such that
$$AP = P \begin{bmatrix}
  M_1 & \varDelta \\
  & M_2
\end{bmatrix} ~\Leftrightarrow~
M := \begin{bmatrix}
  M_1 & \varDelta \\
  & M_2
\end{bmatrix}
= P^{-1} AP.$$
In the decomposition, the top-left block $M_1\in \mathbb{R}^{k\times k}$ represents the action of $A$ on the unstable subspace. Matrix $M_1$, together with  $P_1$, is the ``small  part'' we discussed in the introduction. Note that $M_1$ ($P_1$) is only $k$-by-$k$ ($n$-by-$k$) and thus takes much fewer samples to learn compared to the entire $A$. It is also evident that $M_1$ inherits all unstable eigenvalues of $A$, while $M_2$ inherits all stable eigenvalues.
Finally, we provide the system dynamics in the transformed coordinates. Let $y = [y_1^{\top} ~ y_2^{\top}]^{\top}$ be the coordinate representation of $x$ in the basis of column vectors of $P$ (i.e., $x = Py$). The system dynamics in $y$-coordinates is
\begin{equation}
  \begin{bmatrix}
    y_{1, t+1} \\ y_{2, t+1}
  \end{bmatrix}
  = P^{-1} A P \begin{bmatrix}
    y_{1, t} \\ y_{2, t}
  \end{bmatrix} + P^{-1} Bu_t
  = \begin{bmatrix}
    M_1 & \varDelta \\
    & M_2
  \end{bmatrix} \begin{bmatrix}
    y_{1, t} \\ y_{2, t}
  \end{bmatrix}
  + \begin{bmatrix}
    P_1^{\top} B \\
    P_2^{\top} B
  \end{bmatrix} u_t.\label{eq:dynamics_y}
\end{equation}

\textbf{The $\bm{E_{\mathrm{u}} \oplus E_{\mathrm{s}}}$-decomposition.} In the above ${E_{\mathrm{u}} \oplus E_{\mathrm{u}}^{\perp}}$-decomposition, the subspace $E_{\mathrm{u}}^{\perp}$ is in general \emph{not} invariant with respect to $A$. This can be seen from the top-right $\varDelta$ block in $M$, which represents how much of the state is ``moved'' by $A$ from $E_{\mathrm{u}}^{\perp}$ into $E_{\mathrm{u}}$ in one step. 
The absence of invariant properties in $E_{\mathrm{u}}^{\perp}$ is sometimes inconvenient in the analysis. Hence, in the following, we introduce another invariant decomposition that is used in the proof. Specifically, $\mathbb{R}^n$ can be naturally decomposed into $E_{\mathrm{u}} \oplus E_{\mathrm{s}}$, and further both $E_{\mathrm{u}}$ and $E_{\mathrm{s}}$ are invariant with respect to $A$. We also represent $E_{\mathrm{u}} = \col(Q_1)$ and $E_{\mathrm{s}} = \col(Q_2)$ by their \textit{orthonormal} bases, and define $Q = [Q_1 ~ Q_2]$. Note that, in general, these two subspaces are not orthogonal, we additionally define $Q^{-1} =: [R_1^{\top} R_2^{\top}]^{\top}$. Details are deferred to Appendix \ref{subsec:appendix-decomp-Eu-Es}.

Lastly, we comment that when $A$ is symmetric, the ${E_{\mathrm{u}} \oplus E_{\mathrm{u}}^{\perp}}$- and ${E_{\mathrm{u}} \oplus E_{\mathrm{s}}}$-decompositions are identical because  $E_{\mathrm{u}}^{\perp} = E_{\mathrm{s}}$ in such symmetric cases. While $E_{\mathrm{u}}^{\perp} \neq E_{\mathrm{s}}$ in general cases, the ``closeness'' between $E_{\mathrm{u}}^{\perp}$ and $E_{\mathrm{s}}$ plays an important role in the sample complexity bound in Section~\ref{sec:result}. 
For that purpose, we formally define such ``closeness'' between subspaces in Definition~\ref{definition:xi-close}. We point out that the definition has clear geometric interpretations and leads to connections between the bases of $E_{\mathrm{s}}$ and $E_{\mathrm{u}}^{\perp}$, which is technical and thus deferred to Appendix \ref{subsec:xi-close-geometry}.

\begin{definition}[$\bm{\xi}$-close subspaces]\label{definition:xi-close}
  For $\xi \in (0,1]$, the subspaces $E_{\mathrm{u}}^{\perp} = \col(P_2), E_{\mathrm{s}} = \col(Q_2)$ are called \textbf{$\bm{\xi}$-close} to each other, if and only if $\sigma_{\min}(P_2^{\top} Q_2) > 1-\xi$.
\end{definition}

\subsubsection{\(\tau\)-hop Control}\label{subsec:tau-hop}

This section discusses the design of controller based only on the ``small part'' of $A$, i.e., the $P_1$ and $M_1$ matrices discussed in Section~\ref{subsec:decomposition}, as opposed to the entire $A$ matrix. Note the main goal of this subsection is to introduce the idea of our controller design when $M_1$ and $P_1$ are known without errors, whereas in Section~\ref{subsec:algorithm} we fully introduce Algorithm \ref{alg:main-algorithm} that will learn $M_1$ and $P_1$ before constructing the stabilizing controller. 

As discussed in Section~\ref{subsec:decomposition}, we can view $M_1$ as the ``restriction'' of $A$ onto the unstable subspace $E_{\mathrm{u}}$ (spanned by the basis in $P_1$) and, preferably, it captures all the unstable eigenvalues of $A$. Since only $M_1$ and $P_1$ are known while $M_2$ and $P_2$ are unknown, a simple idea is to ``restrict'' the system trajectory entirely to $E_{\mathrm{u}}$ such that the effect of $A$ is fully captured by $M_1$, the part of $A$ that is known. However, such restriction is not possible because, even if the current state $x_t$ is in $E_{\mathrm{u}}$ (so $Ax_t$ is also in $E_{\mathrm{u}}$), $x_{t+1} = A x_t + B u_t$ is generally not in $E_{\mathrm{u}}$ for non-zero $u_t$. To address this issue, recall that a desirable property of the stable component is that it spontaneously dies out in open loop. Therefore, we propose the following \textit{$\tau$-hop controller} design, where the control input is only injected every $\tau$ steps --- in this way, we let the stable component die out exponentially between two consecutive control injections. Consequently, when we examine the states every $\tau$ steps, we could expect that the trajectory appears approximately ``restricted to'' the unstable subspace $E_{\mathrm{u}}$.

More formally, a $\tau$-hop controller only injects non-zero $u_{t}$ for $t = s\tau$, $s \in \N$. Let $\tilde{x}_s := x_{s\tau}$ and $\tilde{u}_s := u_{s\tau}$ to be the state and input every $\tau$ time steps. We can write the dynamics of the $\tau$-hop control system as $ \tilde{x}_{s+1} = A^{\tau} \tilde{x}_{s} + A^{\tau-1} B \tilde{u}_{s}$. We also let $\tilde{y}_s$ to denote the state under $E_\mathrm{u}\oplus E_{\mathrm{u}}^{\perp}$-decomposition, i.e.  $\tilde{y}_s = P^\top \tilde{x}_s $. Then the state evolution can be written as
\begin{equation}
  \begin{bmatrix}
    \tilde{y}_{1, s+1} \\ \tilde{y}_{2, s+1}
  \end{bmatrix}
  = P^{-1} A^{\tau} P \begin{bmatrix}
    \tilde{y}_{1, s} \\ \tilde{y}_{2, s}
  \end{bmatrix} + P^{-1} A^{\tau-1}B \tilde{u}_s
  = M^{\tau} \begin{bmatrix}
    \tilde{y}_{1, s} \\ \tilde{y}_{2, s}
  \end{bmatrix}
  + \begin{bmatrix}
    P_1^{\top} A^{\tau-1}B \\
    P_2^{\top} A^{\tau-1}B
  \end{bmatrix} \tilde{u}_s,
\end{equation}
where we define $B_{\tau} := P_1^{\top} A^{\tau-1} B$ for simplicity, and
\begin{equation*}
  M^{\tau}
  = \paren{\begin{bmatrix}
    M_1 & \\
    & M_2
  \end{bmatrix} + \begin{bmatrix}
    O & \varDelta \\
    & O
  \end{bmatrix}}^{\tau} 
  = \begin{bmatrix}
    M_1^{\tau} & \sum_{i=0}^{\tau-1} M_1^i \varDelta M_2^{\tau-1-i} \\
    & M_2^{\tau}
  \end{bmatrix}
  =: \begin{bmatrix}
    M_1^{\tau} & \varDelta_{\tau} \\
    & M_2^{\tau}
  \end{bmatrix}.
\end{equation*}
Now we consider a state feedback controller $\tilde{u}_s = K_1 \tilde{y}_{1,s}$ in the $\tau$-hop control system that only acts on the unstable component $\tilde{y}_{1,s}$, the closed-loop dynamics of which can then be written as
\begin{equation}\label{eq:tau-hop-control-dynamics}
  \tilde{y}_{s+1}
  = \begin{bmatrix}
    M_1^\tau + P_1^{\top} A^{\tau-1} B K_1 & \varDelta_\tau \\
    P_2^{\top} A^{\tau-1} B K_1 & M_2^\tau
  \end{bmatrix} \tilde{y}_{s}.
\end{equation}
In \eqref{eq:tau-hop-control-dynamics}, the bottom-left block becomes $ P_2^{\top} A^{\tau-1}BK_1$, which is exponentially small in $\tau$. Therefore, with a properly chosen $\tau$, the closed-loop dynamical matrix in \eqref{eq:tau-hop-control-dynamics} is almost block-upper-triangular with the bottom-right block very close to $O$ (recall that $M_2$ is a stable matrix). As a result, if we select $K_1$ such that $M_1^\tau + P_1^{\top} A^{\tau-1} B K_1$ is stable, then \eqref{eq:tau-hop-control-dynamics} will become stable as well. There are different ways to select such $K_1$, and in this paper, we focus on the simple case that $B$ is an $n$-by-$k$ matrix and $P_1^{\top} A^{\tau-1} B$ is an invertible square matrix (see Assumption \ref{assumption:c-effective-control}), in which case selecting 
\begin{align}
    K_1 = - (P_1^\top A^{\tau-1} B)^{-1} M_1^\tau \label{eq:tau-hop-controller}
\end{align}
will suffice. Note that such a controller design will also need the knowledge of $P_1^\top A^{\tau-1} $, which has the same dimension as $M_1$ (a $k$-by-$k$ matrix) and takes only $O(k)$ additional samples to learn. For the case that $B$ is not $n$-by-$k$, similar controller design can be done (but in a slightly more involved way), and we defer the discussion to Appendix \ref{sec:appendix-transform-k-columns}. 

Finally, we end this section by pointing out that for the case of symmetric $A$, selecting $\tau=1$ should work well. This is because $\Delta_\tau = 0$ in \eqref{eq:tau-hop-control-dynamics} for the symmetric case, and therefore, the matrix in \eqref{eq:tau-hop-control-dynamics} will be triangular even for $\tau = 1$. This will result in a simpler algorithm and controller design, and hence a better sample complexity bound, which we will present as Theorem~\ref{thm:main-theorem-symmetric-A} in Section~\ref{sec:result}.  

\subsection{Algorithm}\label{subsec:algorithm}
Our algorithm, \algname, is divided into 4 stages: (i) learn an orthonormal basis $P_1$ of the unstable subspace $E_{\mathrm{u}}$ (Stage 1); (ii) learn $M_1$, the restriction of $A$ onto the subspace $E_{\mathrm{u}}$ (Stage 2); (iii) learn $B_{\tau} = P_1^{\top} A^{\tau-1} B$ (Stage 3); and (iv) design a controller that seeks to cancel out the ``unstable'' $M_1$ matrix (Stage 4). This is formally described as Algorithm~\ref{alg:main-algorithm} below.

\begin{algorithm}[htbp]
  \caption{\algname: learning a $\tau$-hop stabilizing controller.}\label{alg:main-algorithm}
  \begin{algorithmic}[1]
    \State \textbf{Stage 1: learn the unstable subspace of $\bm{A}$.}
    \State Run the system in open loop for $t_0$ steps for initialization.
    \State Run the system in open loop for $k$ more steps and let $D \leftarrow [x_{t_0+1} ~ \cdots ~ x_{t_0+k}]$.
    \State Calculate $\hat{\varPi}_1 \leftarrow D (D^{\top} D)^{-1} D^{\top}$.
    \State Calculate the top $k$ (normalized) eigenvectors $\hat{v}_1, \cdots \hat{v}_k$ of $\hat{\varPi}_1$, and let $\hat{P}_1 \leftarrow [\hat{v}_1 ~ \cdots ~ \hat{v}_k]$.
    \State \textbf{Stage 2: approximate $\bm{M_1}$ on the unstable subspace.}
    \State Solve the least squares $\hat{M}_1 \leftarrow \arg\min_{M_1\in \mathbb{R}^{k \times k}}\mathcal{L}(M_1) := \sum_{t=t_0+1}^{t_0+k} \snorm{\hat{P}_1^{\top} x_{t+1} - \hat{M}_1 \hat{P}_1^{\top} x_{t}}^2$.
    \State \textbf{Stage 3: restore $\bm{B_{\tau}}$ for $\bm{\tau}$-hop control.}
    \State \textbf{for} $i=1,\cdots,k$ \textbf{do}
      \State\qquad Let the system run in open loop for $\omega$ time steps.
      \State\qquad Run for $\tau$ more steps with initial $u_{t_i} = \alpha \snorm{x_{t_i}} e_i$, where $t_i = t_0+k+i\omega+(i-1)\tau$.
    \State Let $\hat{B}_{\tau} \leftarrow [\hat{b}_1 \cdots \hat{b}_k]$, where the $i$\tsup{th} column $\hat{b}_i \leftarrow \frac{1}{\alpha \snorm{x_{t_i}}} \big( \hat{P}_1^{\top} x_{t_i+\tau} - \hat{M}_1^{\tau} \hat{P}_1^{\top} x_{t_i} \big)$.
    \State \textbf{Stage 4: construct a $\bm{\tau}$-hop stabilizing controller $\bm{K}$.}
    \State Construct the $\tau$-hop stabilizing controller $\hat{K} \leftarrow - \hat{B}_{\tau}^{-1} \hat{M}_1^{\tau} \hat{P}_1^{\top}$.
  \end{algorithmic}
\end{algorithm}

In the remainder of this section we provide detailed descriptions of the four stages in \algname.

\textbf{Stage 1: Learn the unstable subspace of $\bm{A}$.} It suffices to learn an orthonormal basis of $E_{\mathrm{u}}$. We notice that, when $A$ is applied recursively, it will push the state closer to $E_{\mathrm{u}}$. Therefore, when we let the system run in open loop (with control input $u_t \equiv 0$) for $t_0$ time steps, the ratio between the norms of unstable and stable components will be magnified exponentially, and the state lies ``almost'' in $E_{\mathrm{u}}$. As a result, the subspace spanned by the next $k$ states, i.e. the column space of $D := [x_{t_0+1} ~ \cdots ~ x_{t_0+k}]$, is very close to $E_{\mathrm{u}}$. This motivates us to use the orthogonal projector onto $\col(D)$, namely $\hat{\varPi}_1 = D (D^{\top}D)^{-1} D^{\top}$, as an estimation of the projector $\varPi_1 = P_1 P_1^{\top}$ onto $E_{\mathrm{u}}$. Finally, the columns of $\hat{P}_1$ are restored by taking the top $k$ eigenvectors of $\hat{\varPi}_1$ with largest eigenvalues (they should be very close to $1$), which form a basis of the estimated unstable subspace.

\textbf{Stage 2: Learn $\bm{M_1}$ on the unstable subspace.} Recall that $M_1$ is the transition matrix for the $E_{\mathrm{u}}$-component under the $E_{\mathrm{u}} \oplus E_{\mathrm{u}}^{\perp}$-decomposition. Therefore, to estimate $M_1$, we first calculate the coordinates of the states $x_{t_0+1:t_0+k}$ under basis $P_1$; that is, $\hat{y}_{1,t} = \hat{P}_1^\top x_t$, for $t=t_0+1,\ldots,t_0+k$. Then, we use least squares to estimate $M_1$, which minimizes the square loss over $\hat{M}_1$
\begin{equation}\label{eq:stage-2-least-squares}
  \mathcal{L}(\hat{M}_1)
  := \sum_{t=t_0+1}^{t_0+k} \snorm{\hat{y}_{1,t+1} - \hat{M}_1 \hat{y}_{1,t}}^2
  = \sum_{t=t_0+1}^{t_0+k} \snorm{\hat{P}_1^{\top} x_{t+1} - \hat{M}_1 \hat{P}_1^{\top} x_{t}}^2.
\end{equation}
It can be shown that the unique solution to (\ref{eq:stage-2-least-squares}) is $\hat{M}_1 = \hat{P}_1^{\top} A \hat{P}_1$ (see Appendix \ref{sec:appendix-proof-least-squares}).

\textbf{Stage 3: Restore $\bm{B_{\tau}}$ for $\bm{\tau}$-hop control.} In this step, we restore the $B_{\tau}$ that quantifies the ``effective component'' of control inputs restricted to $E_{\mathrm{u}}$ (see Section~\ref{subsec:tau-hop} for detailed discussion). Note that equation (\ref{eq:tau-hop-control-dynamics}) shows
$$y_{1,t_i+\tau} = M^{\tau} y_{1,t_i} + \varDelta_{\tau} y_{2,t_i} + B_{\tau} u_{t_i}.$$
Hence, for the purpose of estimation, we simply ignore the $\varDelta_{\tau}$ term, and take the $i$\tsup{th} column as
$$\hat{b}_i \leftarrow \frac{1}{\snorm{u_{t_i}}} \big( \hat{P}_1^{\top} x_{t_i+\tau} - \hat{M}_1^{\tau} \hat{P}_1^{\top} x_{t_i} \big),$$
where $u_{t_i}$ is parallel to $e_i$ with magnitude $\alpha \snorm{x_{t_i}}$ for normalization. Here we introduce an adjustable constant $\alpha$ to guarantee that the $E_{\mathrm{u}}$-component still constitutes a non-negligible proportion of the state after injecting $u_{t_i}$, so that the iterative restoration of columns could continue.

It is evident that the ignored $\varDelta_{\tau} P_2^{\top} x_{t_i}$ term will introduce an extra estimation error. Since $\varDelta_{\tau}$ contains a factor of $M_1^{\tau-1} \varDelta$ that explodes with respect to $\tau$, this part can only be bounded if $\frac{\snorm{P_2^{\top} x_{t_i}}}{\snorm{x_{t_i}}}$ is sufficiently small. For this purpose, we introduce $\omega$ heat-up steps (running in open loop with 0 control input) to reduce the ratio to an acceptable level, during which time the projection of state onto $E_{\mathrm{u}}^{\perp}$ automatically diminishes over time since $\rho(M_2) = |\lambda_{k+1}| < 1$.

\textbf{Stage 4: Construct a $\bm{\tau}$-hop stabilizing controller $\bm{K}$.} Finally, we can design a controller that cancels out $M_1^{\tau}$ in the $\tau$-hop system. As mentioned in Section~\ref{subsec:tau-hop}, we shall focus on the case where $B$ is an $n$-by-$k$ matrix for the sake of exposition (the case for general $B$ will be discussed in Appendix \ref{sec:appendix-transform-k-columns}). The invertibility of $B_{\tau}$ can be guaranteed under certain conditions (Assumption \ref{assumption:c-effective-control}); further, $\hat{B}_{\tau}$ is also invertible as long as it is close enough to $B_{\tau}$. In this case, the $\tau$-hop stabilizing controller can be simpliy designed as $\hat{K}_1 = - \hat{B}_{\tau}^{-1} \hat{M}_1^{\tau}$ in $y$-coordinates where we replace $B_\tau$ and $M_1$ in \eqref{eq:tau-hop-controller} with their estimates. When we return to the original $x$-coordinates, the controller becomes $\hat{K} = - \hat{B}_{\tau}^{-1} \hat{M}_1^{\tau} \hat{P}_1^{\top}$.
Note that $\hat{K}$ (and $\hat{K}_1$) appears with a hat to emphasize the use of estimated projector $\hat{P}_1$, which introduces an extra estimation error to the final closed-loop dynamics.

It is evident that the algorithm terminates in $t_0 + (1+\omega+\tau)k$ time steps. Therefore, it only suffices to take appropriate parameters so as to guarantee stability and sub-linear time simultaneously.

  \section{Stability Guarantee}\label{sec:result}

In this section, we formally state the assumptions and show the sample complexity of the proposed algorithm in finding a stabilizing controller. Our first assumption is regarding the spectral properties of $A$, which is mild in that we only require the leading eigenvalue $\lambda_1$ to appear without multiplicity or conjugate; eigenvalues with moduli $1$ are prohibited to ensure that eigenspaces are either stable or unstable, which hardly happens in practice and could be eliminated via perturbation.

\begin{assumption}[spectral property]\label{assumption:spectrum-A}
  $A$ is diagonalizable with instability index $k$, with eigenvalues satisfying $|\lambda_1| > |\lambda_2| \geq \cdots \geq |\lambda_k| > 1 > |\lambda_{k+1}| \geq \cdots \geq |\lambda_n|$.
\end{assumption}

Our second assumption is regarding how the initial state is chosen, which again is standard.

\begin{assumption}[initialization]\label{assumption:initialize}
  The initial state of the system is sampled uniformly randomly on the unit hyper-sphere surface in $\R^n$.
\end{assumption}

Lastly, we impose an assumption regarding controllability within the unstable subspace $E_{\mathrm{u}}$.

\begin{assumption}[$\bm{c}$-effective control within unstable subspace]\label{assumption:c-effective-control}
  $B \in \R^{n \times k}$, $\sigma_{\min}(R_1 B) > c \snorm{B}$. 
\end{assumption}
As mentioned in Section~\ref{subsec:tau-hop}, we assume $B$ has $k$ columns for the ease of exposition, and the case for general $B$ is discussed in Appendix \ref{sec:appendix-transform-k-columns}.
In Assumption~\ref{assumption:c-effective-control}, recall matrix $R_1$ that was defined in the $E_{\mathrm{u}}\oplus E_{\mathrm{s}}$-decomposition in Section~\ref{subsec:decomposition}. Intuitively, Assumption \ref{assumption:c-effective-control} characterizes ``effective controllability in $E_{\mathrm{u}}$'' in the following sense: every direction in the unstable subspace receives at least a proportion of $c$ from the influence of any control input. 
This assumption is reasonable in that, if $\sigma_{\min} (R_1 B) \approx 0$, the control input $u$ has to be very large to push the state along the direction corresponding to the smallest singular value, which could induce excessively large control cost.

In the following we present the main performance guarantees for our algorithm.

\begin{theorem}[Main Theorem]\label{thm:main-theorem}
  Given a noiseless LTI system $x_{t+1} = Ax_t + Bu_t$ subject to Assumptions \ref{assumption:spectrum-A}, \ref{assumption:initialize} and \ref{assumption:c-effective-control}, and additionally $|\lambda_1|^2 |\lambda_{k+1}| < |\lambda_k|$, by running \algname\ with parameters
  $$\tau = O(1),~ \omega = O(\ell \log k),~ \alpha = O(1),~ \delta = O(k^{-1/2} |\lambda_1|^{-2\tau}),$$
  that terminates within $O \paren{k \log n}$ time steps, the closed-loop system is exponentially stable with probability $1 - O(k^{-\ell})$ over the initialization of $x_0$ for any $\ell \in \N$.
  Here the big-O notation hides system parameters like $|\lambda_1|$, $|\lambda_{k+1}|$, $\snorm{A}$, $\snorm{B}$, $c$, $\alpha$,
  $\xi$ (recall that $E_{\mathrm{u}}^{\perp}$ and $E_{\mathrm{s}}$ are $\xi$-close),
  $\chi(\hat{L}_{\tau})$ (see Lemma \ref{thm:lemma-block-perturb-spectrum}),
  and $\zeta_{\varepsilon}(\cdot)$ (see Lemma \ref{thm:Gelfand-formula}), and details can be found in equations (\ref{eq:constraint-tau:merge}) through (\ref{eq:constraint-omega:merge}).
\end{theorem}

Theorem~\ref{thm:main-theorem} shows the proposed \algname\ can find a stabilizing controller in $\tilde{O}(k)$ steps, which incurs a state norm of $2^{\tilde{O}(k)}$, significantly smaller than the state-of-the-art $2^{\varTheta(n)}$ in the $k\ll n$ regime.
We will also verify this result numerically in Appendix~\ref{sec:appendix-experiment}. 

\textbf{Discussion on constants.} Curious readers could refer to Appendix~\ref{sec:appendix-proof-general-theorem} for detailed expressions of constants, and for now, we provide a brief overview on how the constants depend on the system parameters. 
It is evident that, for a system with larger $\xi$ (i.e., when $E_{\mathrm{u}}$ and $E_{\mathrm{s}}$ are ``less orthogonal'' to each other) or smaller $c$ (i.e., when it costs more to control the unstable subspace), we will see a larger $\tau$ in (\ref{eq:constraint-tau:merge}), smaller $\alpha$ and $\delta$ in (\ref{eq:constraint-alpha:merge}) and (\ref{eq:constraint-delta:merge}), and a larger $\omega$ in (\ref{eq:constraint-omega:merge}), which altogether incurs a larger constant hidden in the sample complexity.
This is in accordance with our intuition of the state space decomposition and Assumption \ref{assumption:c-effective-control}, respectively.

The bound also relies heavily on the spectral properties of $A$. The constraint $|\lambda_1|^2 |\lambda_{k+1}| < |\lambda_k|$ ensures validity of (\ref{eq:constraint-tau:merge}), which is necessary for cancelling out the combined effect of non-orthogonal subspaces $E_{\mathrm{u}}$ and $E_{\mathrm{s}}$ (resulting in $\varDelta_{\tau}$ in the top-right block) and inaccurate basis $\hat{P}_1$ (resulting in projection error in the bottom-left block) --- a system with larger ratio $|\lambda_1|^2 |\lambda_{k+1}| / |\lambda_k|$ suffers from more severe side-effects, and thus requires a larger $\tau$ and a higher sample complexity. Nevertheless, we believe that this assumption is not essential, and we leave it as future work to relax it.

Another important parameter is the ratio $|\lambda_k| / |\lambda_{k+1}|$ that determines how fast the stable and unstable components become separable in magnitude when the system runs in open loop, which is utilized in the $\omega$ heat-up steps of Stage 3. Consequently, a system with smaller ratio $|\lambda_k| / |\lambda_{k+1}|$ requires a larger $\omega$ (see (\ref{eq:constraint-omega:merge})) and therefore a higher sample complexity.

Despite the generality of Theorem~\ref{thm:main-theorem}, its proof involves technical difficulties. In Theorem~\ref{thm:main-theorem-symmetric-A}, we include results for the special case where $A$ is real symmetric, which, as mentioned in Section~\ref{subsec:algorithm}, leads to a simpler choice of algorithm parameters and a cleaner sample complexity bound.
\begin{theorem}\label{thm:main-theorem-symmetric-A}
  Given a noiseless LTI system $x_{t+1} = Ax_t + Bu_t$ subject to Assumptions \ref{assumption:spectrum-A}, \ref{assumption:initialize} and \ref{assumption:c-effective-control} with symmetric $A$, by running \algname\ with parameters
  $$\tau = 1,~ \omega = 0,~ \alpha = 1,~ \delta = O(k^{-1/2}),$$
  that terminates within $O \paren{k \log n}$ time steps, the closed-loop system is exponentially stable with probability $1$ over the initialization of $x_0$.
  Here the big-O notation hides system parameters like $|\lambda_1|$, $|\lambda_{k+1}|$, $\snorm{A}$, $\snorm{B}$, $c$, and $\chi(\hat{L}_1)$ (see Lemma \ref{thm:lemma-block-perturb-spectrum}), and details can be found in equation (\ref{eq:constant-delta-symmetric-merge}).
\end{theorem}
Although Theorem~\ref{thm:main-theorem-symmetric-A} takes a simpler form, its proof still captures the main insight of our analysis. For this reason, we will use the proof of Theorem~\ref{thm:main-theorem-symmetric-A} as a warm-up example in Section \ref{sec:stability-analysis} before we present the proof ideas of the main Theorem~\ref{thm:main-theorem}.

  \section{Proof Outline}\label{sec:proof-idea}

In this section we will give a high-level overview of the key proof ideas for the main theorems. The full proof details can be found in Appendices \ref{sec:appendix-proof-projector-error}, \ref{sec:appendix-proof-symmetric-theorem} and \ref{sec:appendix-proof-general-theorem} as indicated below.

\textbf{Proof Structure.} The proof is largely divided into two steps. In step 1, we examine how accurate the learner estimates the unstable subspace $E_{\mathrm{u}}$ in Stage 1 and 2. We will show that $\varPi_1$, $P_1$ and $M_1$ can be estimated up to an error of $\delta$ within $t_0 = O(\log \frac{n}{\delta})$ steps. In step 2, we examine the estimation error of $M_1$ and $B_{\tau}$ in Stage 2 and 3 (and thus $\hat{K}_1$), based on which we will eventually show that the $\tau$-hop controller output by Algorithm \ref{alg:main-algorithm} makes the system asymptotically stable via a detailed spectral analysis of 
the dynamical matrix of the closed-loop system.

\textbf{Overview of Step 1.} To upper bound the estimation errors in Stage 1 and 2, we only have to notice that the estimation error of $\varPi_1$ completely captures how well the unstable subspace is estimated, and all other bounds should follow directly from it. The bound on $\snorm{\varPi_1 - \hat{\varPi}_1}$ is shown in Theorem \ref{thm:projector-error}, together with a bound on $\snorm{P_1 - \hat{P}_1}$ as in Corollary \ref{thm:projector-error-corollary}, both of which will be introduced in Section~\ref{sec:estimation-error}.

\textbf{Overview of Step 2.} To analyze the stability of the closed-loop system, we shall first write out the closed-loop dynamics under the $\tau$-hop controller. 
Recall in Section~\ref{subsec:tau-hop} we have defined $\tilde{u}_s, \tilde{x}_s, \tilde{y}_s$ to be the control input, state in $x$-coordinates, and state in $y$-coordinates in the $\tau$-hop control system, respectively. Using these notations, the learned controller can be written as 
$$\tilde{u}_s = \hat{K} \tilde{x}_s
  = \hat{K}_1 \hat{P}_1^{\top} P \tilde{y}_s
  = \begin{bmatrix}
    \hat{K}_1 \hat{P}_1^{\top} P_1 \\
    \hat{K}_1 \hat{P}_1^{\top} P_2
  \end{bmatrix} \tilde{y}_s$$
in $y$-coordinates (as opposed to $\hat{K}_1 \tilde{y}_s$). Therefore, the closed-loop $\tau$-hop dynamics should be
\begin{equation}\label{eq:algorithm-formulation-L-tau}
  \tilde{y}_{s+1}
  = \begin{bmatrix}
    M_1^{\tau} + P_1^{\top} A^{\tau-1}B \hat{K}_1 \hat{P}_1^{\top} P_1 & \varDelta_{\tau} + P_1^{\top} A^{\tau-1}B \hat{K}_1 \hat{P}_1^{\top} P_2 \\
    P_2^{\top} A^{\tau-1}B \hat{K}_1 \hat{P}_1^{\top} P_1 & M_2^{\tau} + P_2^{\top} A^{\tau-1}B \hat{K}_1 \hat{P}_1^{\top} P_2
  \end{bmatrix} \begin{bmatrix}
    \tilde{y}_{1, s} \\ \tilde{y}_{2, s}
  \end{bmatrix}
  =: \hat{L}_{\tau} \tilde{y}_{s},
\end{equation}
which we will show to be asymptotically stable (i.e., $\rho(\hat{L}_{\tau}) < 1$). Note that $\hat{L}_{\tau}$ is given by a 2-by-2 block form, we can utilize the following lemma to assist the spectral analysis of block matrices, the proof of which is deferred to Appendix \ref{sec:appendix-proof-block-estimate-radius}.

\begin{lemma}[block perturbation bound]\label{thm:block-estimate-radius}
  For 2-by-2 block matrices $A$ and $E$ in the form
  $$A = \begin{bmatrix}
    A_1 & O \\
    O & A_2
  \end{bmatrix},~
  E = \begin{bmatrix}
    O & E_{12} \\
    E_{21} & O
  \end{bmatrix},$$
  the spectral radii of $A$ and $A+E$ differ by at most
  $|\rho(A+E) - \rho(A)| \leq \chi(A+E) \snorm{E_{12}} \snorm{E_{21}},$
  where $\chi(A+E)$ is a constant (see Appendix \ref{sec:appendix-proof-block-estimate-radius}).
\end{lemma}

The above lemma provides a clear roadmap in proving $\rho(\hat{L}_{\tau}) < 1$. First, we need to guarantee stability of the diagonal blocks --- the top-left block is stable because $\hat{K}_1$ is designed to (approximately) eliminate it to zero (which requires the estimation error bound on $B_{\tau}$), and the bottom-right block is stable because it is almost $M_2^{\tau}$ with a negligible error induced by inaccurate projection. Then, we need to upper-bound the norms of off-diagonal blocks via careful estimation of factors appearing in these blocks. 

The rest of this section just follows the above proof structure. We will first present the estimation error results (Step 1) in Section \ref{sec:estimation-error}, and proceed to show stability guarantee (Step 2) in Section \ref{sec:stability-analysis}.

\subsection{Step 1: Estimation Error of the Unstable Subspace}\label{sec:estimation-error}
As stated above, it is expected that the bound of the top-left block relies heavily on the estimation error of $P_1$. The major concern of this section is to show that the desired estimation precision can be achieved in acceptible time --- specifically, we want it to be in the order of $O(\log n)$. Following the procedure of our algorithm, we will first bound the estimation error of $\varPi_1$, as in Theorem \ref{thm:projector-error}.

\begin{theorem}\label{thm:projector-error}
  For a noiseless linear dynamical system $x_{t+1} = Ax_t$, let $E_{\mathrm{u}}$ be the unstable subspace of $A$, $k = \dim E_{\mathrm{u}}$ be the instability index of the system, and $\varPi_1$ be the orthogonal projector onto subspace $E_{\mathrm{u}}$. Then for any $\varepsilon > 0$, by running Stage 1 of Algorithm \ref{alg:main-algorithm} with an arbitrary initial state that terminates in $(t_0 + k)$ time steps, where
  $$t_0 = O \paren{\log \frac{n}{\varepsilon}},$$
  with probability $1$ the matrix $D^{\top} D$ is invertible \textit{(where $D = [x_{t_0+1} ~ \cdots ~ x_{t_0+k}]$)}, in which case we shall obtain an estimated $\hat{\varPi}_1 = D (D^{\top} D)^{-1} D^{\top}$ with error
  $\snorm{\hat{\varPi}_1 - \varPi_1} < \varepsilon$.
\end{theorem}

The proof of Theorem \ref{thm:projector-error} is deferred to Appendix \ref{sec:appendix-proof-projector-error} due to limited length. The main idea is to diagonalize $A$ and write the open-loop system dynamics using the basis formed by the eigenvectors of $A$. Then, we provide an explicit expression for $\hat{\varPi}_1$ and $\varPi_1$, based on which we can bound the error. 
To further derive a bound for $\snorm{\hat{P}_1 - P_1}$, one only needs to notice that norms are preserved under orthonormal coordinate transformations, so it only suffices to find a specific pair of bases of $E_{\mathrm{u}}^{\perp}$ and $E_{\mathrm{s}}$ that are close to each other --- and the pair of bases formed by principle vectors (see Appendix \ref{sec:appendix-decomp-principle-angle}) is exactly what we want. This leads to Corollary~\ref{thm:projector-error-corollary} that is repeatedly used in subsequent proofs, the proof of which can be also found in Appendix~\ref{sec:appendix-proof-projector-error}.

\begin{corollary}\label{thm:projector-error-corollary}
  Under the premises of Theorem \ref{thm:projector-error}, for any orthonormal basis $\hat{P}_1$ of $\col(\hat{\varPi}_1)$ (where $\hat{\varPi_1}$ is obtained by Algorithm \ref{alg:main-algorithm}, there exists a corresponding orthonormal basis $P_1$ of $\col(\varPi_1)$, such that
  $\snorm{\hat{P}_1 - P_1} < \sqrt{2k} \varepsilon =: \delta$, $
    \snorm{\hat{M}_1 - M_1} < 2 \snorm{A} \delta$. 
\end{corollary}

\subsection{Step 2: Stability Analysis}\label{sec:stability-analysis}
We first consider a warm-up case where $A$ is symmetric, and then proceed to the general case. 

\textbf{Warm-up: symmetric case.} In this case, the eigenvectors of $A$ are mutually orthogonal, which guarantees $E_{\mathrm{u}}^{\perp} = E_{\mathrm{s}}$ (i.e., they are $0$-close to each other) and thus $\varDelta = O$. This allows us to select $\tau = 1$, $\omega = 0$ and $\alpha = 1$, and the closed-loop dynamical matrix simplifies to
\begin{equation}
  \hat{L}_1 = \begin{bmatrix}
    M_1 + P_1^{\top} B \hat{K}_1 \hat{P}_1^{\top} P_1
      & P_1^{\top} B \hat{K}_1 \hat{P}_1^{\top} P_2 \\
    P_2^{\top} B \hat{K}_1 \hat{P}_1^{\top} P_1
      & M_2 + P_2^{\top} B \hat{K}_1 \hat{P}_1^{\top} P_2
  \end{bmatrix}.
\end{equation}
The norm of the top-left block is in the order of $O(\delta)$ based on the estimation error bound (see Theorem \ref{thm:estimation-error-B-symmetric-A}) $\snorm{\hat{B}_1 - B_1} = O(\sqrt{k} \delta)$, which characterizes how well the controller can eliminate the unstable component. The spectrum of the bottom-right block can be viewed as a perturbation (note that $\snorm{\hat{P}_1^{\top} P_2} = O(\delta)$ is small by Proposition \ref{prop:norm-bound-projection-error}) to a stable matrix $M_2$ (recall $\rho(M_2) = |\lambda_{k+1}|$), which should also be stable as long as $\delta$ is small enough. Meanwhile, the top-right block is also approximately zero, since only projection error contributes to the top-right block (again $\snorm{\hat{P}_1^{\top} P_2} = O(\delta)$). The above observations together show that $\hat{L}_1$ is in the order of
\begin{equation}
  \hat{L}_{1} = \begin{bmatrix}
    O(\delta) & O(\delta) \\
    O(1) & |\lambda_{k+1}| + O(\delta)
  \end{bmatrix},
\end{equation}
which is almost lower-triangular. Therefore, we can apply the block perturbation bound to bound the spectrum of $\hat{L}_1$. All relevant proofs are deferred to Appendix \ref{sec:appendix-proof-symmetric-theorem} due to limited length.

\textbf{General case.} For the general case, the analysis becomes more challenging for two reasons: on the one hand, we have to apply $\tau$-hop control with $\tau$ possibly larger than $1$, which potentially increases the norm of $B_{\tau}$ and $\hat{K}_1$; on the other hand, the top-right corner will no longer be $O(\delta)$ with a non-zero $\varDelta$ (in fact, $\varDelta_{\tau}$ is in the order of $|\lambda_1|^{\tau}$ that grows exponentially with respect to $\tau$). To settle these issues, we first introduce two key observations on bounds of major factors:
\begin{enumerate}
  \item For an arbitrary matrix $X$, although $\snorm{X}$ might be significantly larger than $\rho(X)$, we always have $\snorm{X^t} = O(\rho(X)^t)$ when $t$ is large enough. This is formally proven as Gelfand's Formula (see Lemma \ref{thm:Gelfand-formula}), and helps to establish bounds like $\snorm{M_1} = O(|\lambda_1|^{\tau})$, $\snorm{M_2} = O(|\lambda_{k+1}|^{\tau})$, $\snorm{\varDelta_{\tau}} = O(|\lambda_1|^{\tau})$, $\snorm{P_2^{\top} A^{\tau-1}} = O(|\lambda_{k+1}|^{\tau})$, and $\snorm{\hat{M}_1^{\tau} - M_1^{\tau}} = O(|\lambda_1|^{\tau} \delta)$.
  
  \item When the system runs with 0 control inputs for a long period (specifically, for $\omega$ time steps), eventually we will see the unstable component expanding and the stable component shrinking, and consequently $\frac{\snorm{P_2^{\top} A^{\omega} x}}{\snorm{A^{\omega} x}} = O(|\lambda_k|^{-\omega})$. This cancels out the exponentially exploding $\snorm{\varDelta_{\tau}}$, and helps to establish the estimation bound $\snorm{\hat{B}_{\tau} - B_{\tau}} = O(|\lambda_1|^{\tau} \delta)$. 
\end{enumerate}
With these in hand, we are ready to upper bound the norms of the blocks in $\hat{L}_{\tau}$:
\begin{enumerate}
  \item \textit{The top-left and bottom-right blocks}: similar to the warm-up case, only to note that dynamical matrices are lifted to their $\tau$\tsup{th} power, and thus $\snorm{\hat{B}_{\tau} - B_{\tau}}$ carries an additional factor of $|\lambda_1|^{\tau}$.
  
  \item \textit{The bottom-left block}: $P_2^{\top} A^{\tau-1}$ contributes an $O(|\lambda_{k+1}|^{\tau})$ factor that decays exponentially, while $\hat{K}_1$ contributes an $O(|\lambda_1|^{\tau})$ factor that explodes exponentially. The overall bound is in the order of $O(|\lambda_1 \lambda_{k+1} / \lambda_k|^{\tau})$, and decays with respect to $\tau$ if $|\lambda_1 \lambda_{k+1}| < 1$.
  
  \item \textit{The top-right block}: the first term is in the order of $O(|\lambda_1|^{\tau})$, and the second term is in the order of $O(|\lambda_1 \lambda_{k+1} / \lambda_k|^{\tau} \delta)$. This block is in the order of $O(|\lambda_1|^{\tau})$ when $\delta$ is small enough.
\end{enumerate}
Therefore, the closed-loop dynamical matrix is actually in the order of
\begin{equation}
  \hat{L}_{\tau} = \begin{bmatrix}
  O(|\lambda_1|^{2\tau} \delta) & O\big( |\lambda_1|^{\tau} + |\lambda_1 \lambda_{k+1} / \lambda_k|^{\tau} \delta \big) \\
  O(|\lambda_1 \lambda_{k+1} / \lambda_k|^{\tau}) & O\big( |\lambda_{k+1}|^{\tau} + |\lambda_1 \lambda_{k+1}|^{\tau} \delta \big)
\end{bmatrix}.
\end{equation}
Finally, by Lemma \ref{thm:block-estimate-radius}, asymptotic stability is guaranteed when $|\lambda_1|^2 |\lambda_{k+1}| < |\lambda_k|$ (i.e., the norm of the bottom-left block decays faster than the norm of the top-right block grows), in which case we can set $\tau$ to be some constant determined by $A$ and $B$, and $\delta$ in the order of $O(|\lambda_1|^{-2\tau})$.

The proofs in this subsection are deferred to Appendix \ref{sec:appendix-proof-general-theorem} due to limited length.

  \bibliographystyle{plainnat}
  \bibliography{biblio}
  
  \newpage
  \appendix
  \begin{center}
    \Large\textbf{Appendices}
  \end{center}
  
  \section{Decomposition of the State Space}\label{sec:appendix-decomp-principle-angle}

\subsection{The \texorpdfstring{\titlemath{\bm{E_{\mathrm{u}} \oplus E_{\mathrm{s}}}}}{E\_u⊕E\_s}-decomposition}\label{subsec:appendix-decomp-Eu-Es}

It is evident that the following two subspaces of $\R^n$ are invariant with respect to $A$, namely
\begin{gather*}
  E_{\mathrm{u}} := \bigoplus_{i \leq k} E_i,~
  E_{\mathrm{s}} := \bigoplus_{i > k} E_i
\end{gather*}
which we refer to as the \textit{unstable subspace} and the \textit{stable subspace} of $A$, respectively. Since the eigenspaces $E_i$ sum to the whole $\R^n$ space, one natural decomposition is $\R^n = E_{\mathrm{u}} \oplus E_{\mathrm{s}}$; accordingly, each state can be uniquely decomposed as $x = x_{\mathrm{u}} + x_{\mathrm{s}}$, where $x_{\mathrm{u}} \in E_{\mathrm{u}}$ is called the \textit{unstable component}, and $x_{\mathrm{s}} \in E_{\mathrm{s}}$ is called the \textit{stable component}.

We also decompose $A$ based on the $E_{\mathrm{u}} \oplus E_{\mathrm{s}}$-decomposition. Suppose $E_{\mathrm{u}}$ and $E_{\mathrm{s}}$ are represented by their \textit{orthonormal} bases $Q_1 \in\R^{n \times k}$ and $Q_2 \in \R^{n \times (n-k)}$, respectively, namely
$$E_{\mathrm{u}} = \col(Q_1),~
  E_{\mathrm{s}} = \col(Q_2).$$
Let $Q = [Q_1 ~ Q_2]$ (which is invertible as long as $A$ is diagonalizable), and let $R = [R_1^{\top} ~ R_2^{\top}]^{\top} := Q^{-1}$. Further, let $\varPi_{\mathrm{u}} := Q_1 R_1$ and $\varPi_{\mathrm{s}} = Q_2 R_2$ be the \textit{oblique} projectors onto $E_{\mathrm{u}}$ and $E_{\mathrm{s}}$ (along the other subspace), respectively. Since $E_{\mathrm{u}}$ and $E_{\mathrm{s}}$ are both invariant with regard to $A$, we know there exists $N_1 \in \R^{k \times k}, N_2 \in \R^{(n-k) \times (n-k)}$, such that
$$AQ = Q \begin{bmatrix}
  N_1 & \\
  & N_2
\end{bmatrix} ~\Leftrightarrow~
N := \begin{bmatrix}
  N_1 & \\
  & N_2
\end{bmatrix}
= RAQ.$$
Let $z = [z_1^{\top} ~ z_2^{\top}]^{\top}$ be the coordinate representation of $x$ in the basis $Q$ (i.e., $x = Qz$). The system dynamics in $z$-coordinates can be expressed as
$$\begin{bmatrix}
    z_{1, t+1} \\ z_{2, t+1}
  \end{bmatrix}
  = R A Q \begin{bmatrix}
    z_{1, t} \\ z_{2, t}
  \end{bmatrix} + R Bu_t
  = \begin{bmatrix}
    N_1 & \\
    & N_2
  \end{bmatrix} \begin{bmatrix}
    z_{1, t} \\ z_{2, t}
  \end{bmatrix}
  + \begin{bmatrix}
    R_1 B \\
    R_2 B
  \end{bmatrix} u_t.$$
The major advantage of this decomposition is that the dynamical matrix in $z$-coordinate is block diagonal, so it would be simpler to study the behavior of the open-loop system.

\subsection{Geometric Interpretation: Principle Angles}\label{subsec:xi-close-geometry}

\begin{wrapfigure}{r}{5cm}
  \centering
  \vspace*{-24pt}
  \begin{tikzpicture}[scale=1.4]
  \fill[fill=LimeGreen, fill opacity=0.3] (-20pt, 0pt) -- (20pt, 0pt) -- (5pt, -40pt) -- (-35pt, -40pt) -- cycle;
  \fill[fill=ProcessBlue, fill opacity=0.3] (-20pt, 0pt) -- (20pt, 0pt) -- (3.5pt, 44pt) -- (-36.5pt, 44pt) -- cycle;
  \fill[fill=LimeGreen, fill opacity=0.3] (-20pt, 0pt) -- (20pt, 0pt) -- (35pt, 40pt) -- (-5pt, 40pt) -- cycle;
  \fill[fill=ProcessBlue, fill opacity=0.3] (-20pt, 0pt) -- (20pt, 0pt) -- (36.5pt, -44pt) -- (-3.5pt, -44pt) -- cycle;
  
  \draw[draw=DarkGreen, line width=1.0pt, draw opacity=0.5] (-5pt, -40pt) -- (-35pt, -40pt) -- (-5pt, 40pt) -- (35pt, 40pt) -- (20pt, 0pt);
  \draw[draw=DarkGreen, line width=1.0pt, draw opacity=0.5, dashed] (20pt, 0pt) -- (5pt, -40pt) -- (-5pt, -40pt);
  \draw[draw=DarkBlue, line width=1.0pt, draw opacity=0.5] (5pt, 40pt) -- (3.5pt, 44pt) -- (-36.5pt, 44pt) -- (-3.5pt, -44pt) -- (36.5pt, -44pt) -- (20pt, 0pt);
  \draw[draw=DarkBlue, line width=1.0pt, draw opacity=0.5, dashed] (5pt, 40pt) -- (20pt, 0pt);
  
  \draw[draw=Crimson, line width=1.0pt, draw opacity=0.5] (-40pt, -32pt) -- (-28.57pt, -22.86pt);
  \draw[draw=Crimson, line width=1.0pt, draw opacity=0.5, dashed] (-28.57pt, -22.86pt) -- (0pt, 0pt);
  \draw[draw=Crimson, line width=1.0pt, draw opacity=0.5] (0pt, 0pt) -- (40pt, 32pt);
  
  \node at (28pt, -39pt) {\small\color{DarkBlue} $E_{\mathrm{u}}^{\perp}$};
  \node at (-30pt, -36pt) {\small\color{DarkGreen} $E_{\mathrm{s}}$};
  \node at (45pt, 31pt) {\small\color{Crimson} $E_{\mathrm{u}}$};
  
  \draw[draw=DarkBlue, line width=0.5pt, ->, >={Stealth}] (0pt, 0pt) -- (-4.5pt, 12pt);
  \draw[draw=DarkBlue, line width=0.5pt, ->, >={Stealth}] (0pt, 0pt) -- (12.8pt, 0pt);
  \draw[draw=DarkGreen, line width=0.5pt, ->, >={Stealth}] (0pt, 0pt) -- (4.5pt, 12pt);
  \draw[draw=Crimson, line width=0.5pt, ->, >={Stealth}] (0pt, 0pt) -- (11pt, 8.8pt);
  
  \node at (-5pt, 15pt) {\footnotesize\color{DarkBlue} $\alpha_1$};
  \node at (5pt, 15pt) {\footnotesize\color{DarkGreen} $\beta_1$};
  \node at (12.8pt, -4pt) {\footnotesize\color{DarkBlue} $\alpha_2$};
  \node at (12.8pt, -10pt) {\footnotesize\color{DarkGreen} $(\beta_2)$};
  
  \draw[draw=black, line width=0.5pt] (-17.2pt, 7.5pt) arc (69.44:110.56:8pt);
  \node at (-20pt, 12pt) {\footnotesize $\theta_1$};
\end{tikzpicture}
  \vspace*{-20pt}
\end{wrapfigure}
Before going any further, we emphasize that Definition \ref{definition:xi-close} is well-defined by itself, since singular values are preserved under orthonormal transformations.

It might seem unintuitive to interpret $\sigma_{\min}(P_2^{\top} Q_2)$ in Definition \ref{definition:xi-close} as a measure of ``closeness''. However, this is closely related to the \textit{principle angles} between subspaces that generalize the standard angle measures in lower dimensional cases. More specifically, we can recursively define the $i$\tsup{th} principle angle $\theta_i$ \textit{($i=1,\cdots,n-k$)} as
\begin{equation}\label{eq:principle-angle-definition}
  \theta_{i} := \min \braced{\arccos \paren{\frac{\angled{x,y}}{\snorm{x}\snorm{y}}} ~\middle|~
  \begin{gathered}
    x \in E_{\mathrm{u}}^{\perp},~ x \perp \Span(x_1, \cdots, x_{i-1}); \\
    y \in E_{\mathrm{s}},~ y \perp \Span(y_1, \cdots, y_{i-1}).
  \end{gathered}}
  =: \angle (x_i, y_i),
\end{equation}
where $x_i$ and $y_i$ \textit{($i=1,\cdots,n-k$)} are referred to as the $i$\tsup{th} principle vectors accordingly. Meanwhile, let $P_2^{\top} Q_2 = U \varSigma V^{\top}$ be the singular value decomposition (SVD), where $\varSigma = \diag(\sigma_1, \cdots, \sigma_{n-k})$ and $\sigma_1 \geq \cdots \geq \sigma_{n-k}$. Then by an equivalent recursive characterization of singular values, we have
$$\sigma_i
  = \max_{\begin{subarray}{c} \snorm{x}=\snorm{y}=1 \\ \forall j<i:~ x \perp x_j,~ y \perp y_j \end{subarray}} x^{\top} P_2^{\top} Q_2 y
  =: \bar{x}_i^{\top} P_2^{\top} Q_2 \bar{y}_i.$$
Since $P_2$ and $Q_2$ are orthonormal, $\bar{x}_i$ and $\bar{y}_i$ can be regarded as coordinate representations of $x_i = P_2 \bar{x}_i$ and $y_i = Q_2 \bar{y}_i$, and it can be easily verified that $x_i$ and $y_i$ defined in this way are exactly the minimizers in (\ref{eq:principle-angle-definition}). Hence we conclude that $\sigma_i = \cos \theta_i$. Therefore, $E_{\mathrm{u}}^{\perp}$ and $E_{\mathrm{s}}$ are $\xi$-close if and only if the all principle angles between $E_{\mathrm{u}}^{\perp}$ and $E_{\mathrm{s}}$ lie in the interval $[0, \arccos(1-\xi)]$; the above argument also shows that we can find orthonormal bases for $E_{\mathrm{u}}^{\perp}$ and $E_{\mathrm{s}}$ so that corresponding vectors form exactly the principle angles.

\subsection{Characterization of \titlemath{\xi}-close Subspaces}

It is naturally expected that the geometric interpretation should inspire more relationships among $P_1 = Q_1, P_2, Q_2, R_1, R_2$ and $N_2$. We would like to emphasize that $P_1$, $P_2$ and $Q_1$ are not confined to bases consisting of eigenvectors (since they are even not necessarily orthonormal). Meanwhile, since they are only used in the stability guarantee proof, we are granted the freedom to select any orthonormal bases. For simplicity, we will stick to the convention that $P_1 = Q_1$ (and thus $M_1 = N_1$). Further, in Lemma \ref{thm:lemma-xi-close-interpretation}, such freedom is utilized to establish fundamental relationships between the bases in the above two decompositions. The results are concluded as follows.

\begin{lemma}\label{thm:lemma-xi-close-interpretation}
  Suppose $E_{\mathrm{u}}^{\perp}$ and $E_{\mathrm{s}}$ are $\xi$-close. Then we shall select $P_2$ and $Q_2$ such that
  \begin{enumerate}
    \item $\sigma_{\min}(P_2^{\top} Q_2) \geq 1-\xi$, $\snorm{P_1^{\top} Q_2} \leq \sqrt{2\xi}$, $\snorm{P_2-Q_2} \leq \sqrt{2\xi}$.
    \item $\snorm{R_2} \leq \frac{1}{1-\xi}$, $\snorm{N_2} \leq \frac{1}{1-\xi} \snorm{A}$. 
    \item $\snorm{P_1^{\top} - R_1} \leq \frac{\sqrt{2\xi}}{1-\xi}$, $\snorm{R_1} \leq \frac{\sqrt{2\xi}}{1-\xi} + 1$.
    \item $\snorm{\varDelta} \leq \frac{2-\xi}{1-\xi} \sqrt{2\xi} \snorm{A}$.
  \end{enumerate}
\end{lemma}

\begin{proof}
  (1) Following the above interpretation, take arbitrary orthonormal bases $\bar{P}_2$ and $\bar{Q}_2$ of $E_{\mathrm{u}}^{\perp}$ and $E_{\mathrm{s}}$, respectively, and let $\bar{P}_2^{\top} \bar{Q}_2 = U \varSigma V^{\top}$ be the SVD, which translates to
  $$(\bar{P}_2 U)^{\top} (\bar{Q}_2 V) = \varSigma =: \diag(\sigma_1,\cdots,\sigma_{n-k}).$$
  Since $U$ and $V$ are orthonormal matrices, the columns of $\bar{P}_2 U$ and $\bar{Q}_2 V$ also form orthonormal bases of $E_{\mathrm{u}}^{\perp}$ and $E_{\mathrm{s}}$, respectively. Then $\xi$-closeness basically says that there exist a basis $\sbraced{\alpha_1, \cdots, \alpha_{n-k}}$ for $E_{\mathrm{u}}^{\perp}$, and a basis $\sbraced{\beta_1, \cdots, \beta_{n-k}}$ for $E_{\mathrm{s}}$ (both are assumed to be orthonormal), such that
  $$\sangled{\alpha_i, \beta_j} = \delta_{ij} \sigma_i
    = \begin{cases}
      \sigma_i \geq 1-\xi & \textrm{for any}~ i = j \\
      0 & \textrm{for any}~ i \neq j
    \end{cases},$$
  and we also have $\varPi_2 \beta_i = \sigma_i \alpha_i$ and $\varPi_1 \alpha_i = \sigma_i \beta_i$ (recall that $\varPi_1, \varPi_2$ are orthogonal projectors onto subspaces $E_{\mathrm{u}}, E_{\mathrm{u}}^{\perp}$, respectively). Therefore, without loss of generality, we shall always select $P_2 = [\alpha_1 ~ \cdots ~ \alpha_{n-k}]$ and $Q_2 = [\beta_1 ~ \cdots ~ \beta_{n-k}]$, such that $P_2^{\top} Q_2 = \diag(\sigma_1,\cdots,\sigma_{n-k})$, and
  $$\sigma_{\min}(P_2^{\top} Q_2) = \min_{i} |\sigma_i| \geq 1-\xi.$$
  Equivalently speaking, for any $\beta = Q_2 \eta \in E_{\mathrm{s}}$, we have (note that $\snorm{\eta} = \snorm{\beta}$)
  $$\snorm{P_2^{\top} \beta}
    = \snorm{P_2^{\top} Q_2 \eta}
    \geq \sigma_{\min}(P_2^{\top} Q_2) \snorm{\eta}
    \geq (1-\xi) \snorm{\beta},$$
  and consequently,
  $$\snorm{P_1^{\top} Q_2 \eta}
    = \snorm{P_1^{\top} \beta}
    = \sqrt{\snorm{\beta}^2 - \snorm{P_2^{\top} \beta}^2}
    \leq \sqrt{2\xi} \snorm{\beta}
    = \sqrt{2\xi} \snorm{\eta},$$
  which further shows $\snorm{P_1^{\top} Q_2} \leq \sqrt{2\xi}$. To bound $\snorm{P_2-Q_2}$, by definition we have
  \begin{align*}
    \snorm{P_2-Q_2}
    &= \max_{\snorm{\eta}=1} \snorm{(P_2-Q_2)\eta} 
    = \max_{\snorm{\eta}=1} \norm{\sum_{i} \eta_i (\alpha_i-\beta_i)} \\
    &= \max_{\snorm{\eta}=1} \sqrt{\sum_{i,j} \eta_i \eta_j (\alpha_i-\beta_i)^{\top} (\alpha_j-\beta_j)} \\
    &= \max_{\snorm{\eta}=1} \sqrt{\sum_{i} 2 (1-\mu_i) \eta_i^2} \\
    &\leq \max_{\snorm{\eta}=1} \sqrt{2\xi \sum_{i} \eta_i^2} 
    = \sqrt{2\xi}.
  \end{align*}
  Here $\eta = [\eta_1, \cdots, \eta_{n-k}]$ is an arbitrary vector in $\R^{n-k}$.
  
  (2) By definition, $I = QR = Q_1 R_1 + Q_2 R_2$. Also recall that $P_1 = Q_1$, so $P_1^{\top} Q_1 = I$ and $P_2^{\top} Q_1 = O$. Then by left-multiplying $P_2^{\top}$ to the equality, we have
  $$P_2^{\top} = P_2^{\top} Q_1 R_1 + P_2^{\top} Q_2 R_2 = P_2^{\top} Q_2 R_2,$$
  which further shows
  $$\snorm{R_2}
    = \snorm{(P_2^{\top} Q_2)^{-1} P_2^{\top}}
    \leq \snorm{(P_2^{\top} Q_2)^{-1}}
    = \frac{1}{\sigma_{\min}(P_2^{\top} Q_2)}
    \leq \frac{1}{1-\xi}.$$
  Therefore, since $N_2 = R_2 A Q_2$, we have
  $$\snorm{N_2}
    = \snorm{R_2 A Q_2}
    \leq \snorm{R_2} \snorm{A} \snorm{Q_2}
    \leq \frac{1}{1-\xi} \snorm{A}.$$
  
  (3) Similarly, by left-multiplying $P_1^{\top}$ to the equality, we have
  $$P_1^{\top} = P_1^{\top} Q_1 R_1 + P_1^{\top} Q_2 R_2 = R_1 + P_1^{\top} Q_2 R_2,$$
  which further shows
  $$\snorm{P_1^{\top} - R_1}
    = \snorm{P_1^{\top} Q_2 R_2}
    \leq \snorm{P_1^{\top} Q_2} \snorm{R_2}
    \leq \frac{\sqrt{2\xi}}{1-\xi},$$
  and therefore $\snorm{R_1} \leq \snorm{P_1^{\top} - R_1} + \snorm{P_1^{\top}} = 1 + \frac{\sqrt{2\xi}}{1-\xi}$.
  
  (4) A combination of the above results gives
  \begin{align*}
    \snorm{\varDelta}
    &= \snorm{P_1^{\top} A P_2}
    = \snorm{P_1^{\top} A P_2 - R_1 A Q_2} \\
    &\leq \snorm{P_1^{\top} A (P_2-Q_2)} + \snorm{(P_1^{\top}-R_1) A Q_2} \\
    &\leq \snorm{P_1^{\top}} \snorm{A} \snorm{P_2-Q_2} + \snorm{P_1^{\top}-R_1} \snorm{A} \snorm{Q_2} \\
    &\leq \snorm{A} \sqrt{2\xi} + \frac{\sqrt{2\xi}}{1-\xi} \snorm{A}
    = \frac{2-\xi}{1-\xi} \sqrt{2\xi} \snorm{A}.
  \end{align*}
  This completes the proof.
\end{proof}
  \section{Solution to the Least Squares Problem in Stage 2}\label{sec:appendix-proof-least-squares}

Lemma \ref{thm:least-squares} gives the explicit form for the solution to the least squares problem (see Algorithm \ref{alg:main-algorithm}).

\begin{lemma}\label{thm:least-squares}
  Given $D := [x_{t_0+1} ~ \cdots ~ x_{t_0+k}]$ and $\hat{P}_1 \hat{P}_1^{\top} = \hat{\varPi}_1 = D (D^{\top}D)^{-1} D^{\top}$, the solution
  $$\hat{M}_1 = \argmin_{M_1} \sum_{t=t_0+1}^{t_0+k} \snorm{\hat{P}_1^{\top} x_{t+1} - M_1 \hat{P}_1^{\top} x_{t}}^2$$
  is uniquely given by $\hat{M}_1 = \hat{P}_1^{\top} A \hat{P}_1$.
\end{lemma}

\begin{proof}
  Here we assume by default that the summation over $t$ sums from $t_0+1$ to $t_0+k$. Since $M_1$ is a stationary point of $\mathcal{L}$, for any $\varDelta$ in the neighbourhood of $O$, we have
  \begin{align*}
    0 \leq \mathcal{L}(M_1 + \varDelta) - \mathcal{L}(M_1)
    &= \sum_{t} \snorm{\hat{y}_{1,t+1} - M_1 \hat{y}_{1,t} - \varDelta \hat{y}_{1,t}}^2 - \sum_{t} \snorm{\hat{y}_{1,t+1} - M_1 \hat{y}_{1,t}}^2 \\
    &= \sum_{t} \angled{\varDelta \hat{y}_{1,t}, \hat{y}_{1,t+1} - M_1\hat{y}_{1,t}} + O(\snorm{\varDelta}^2) \\
    &= \sum_{t} \tr\paren{\hat{y}_{1,t}^{\top} \varDelta^{\top} (\hat{y}_{1,t+1} - A\hat{y}_{1,t})} + O(\snorm{\varDelta}^2) \\
    &= \sum_{t} \tr\paren{\varDelta^{\top} (\hat{y}_{1,t+1} - M_1\hat{y}_{1,t}) \hat{y}_{1,t}^{\top}} + O(\snorm{\varDelta}^2) \\
    &= \tr\paren{\varDelta^{\top} \sum_{t} (\hat{y}_{1,t+1} - M_1\hat{y}_{1,t}) \hat{y}_{1,t}^{\top}} + O(\snorm{\varDelta}^2).
  \end{align*}
  Since it always holds for any $\varDelta$, we must have
  $$\sum_{t} (\hat{y}_{1,t+1} - M_1\hat{y}_{1,t}) \hat{y}_{1,t}^{\top}
    ~\Leftrightarrow~ M_1 \sum_{t} \hat{y}_{1,t} \hat{y}_{1,t}^{\top} = \sum_{t} \hat{y}_{1,t+1} \hat{y}_{1,t}^{\top}.$$
  Plugging in $\hat{y}_{1,t} = \hat{P}_1^{\top} x_t$ and $\hat{y}_{1,t+1} = \hat{P}_1^{\top} A x_t$, we further have
  $$M_1 \hat{P}_1^{\top} X \hat{P}_1
    = M_1 \sum_{t} \hat{P}_1^{\top} x_t x_t^{\top} \hat{P}_1
    = \sum_{t} \hat{P}_1^{\top} A x_t x_t^{\top} \hat{P}_1
    = \hat{P}_1^{\top} A X \hat{P}_1,$$
  where $X := \sum_{t} x_t x_t^{\top} = DD^{\top}$. Since the columns of $\hat{P}_1$ form an orthonormal basis of $\hat{E}_{\mathrm{u}}$, for any $x \in \hat{E}_{\mathrm{u}}$, $\hat{P}_1^{\top} x$ is the coordinate of $x$ under that basis. The columns of $D$ are linearly independent, so the columns of $\hat{P}_1^{\top} D$ are also linearly independent, which further yields
  $$\rank(\hat{P}_1^{\top} X \hat{P}_1)
    = \rank \big( (\hat{P}_1^{\top} D) (\hat{P}_1^{\top} D)^{\top} \big)
    = \rank(\hat{P}_1^{\top} D)
    = k.$$
  Therefore, $\hat{P}_1^{\top} X \hat{P}_1$ is invertible, and $M_1$ is explicitly given by
  $$M_1 = (\hat{P}_1^{\top} A X \hat{P}_1) (\hat{P}_1^{\top} X \hat{P}_1)^{-1}.$$
  Note that $\hat{\varPi}_1 = \hat{P}_1 \hat{P}_1^{\top}$ is the projector onto subspace $\col(D)$, we must have
  $$\hat{P}_1 \hat{P}_1^{\top} X = (\hat{\varPi}_1 D) D^{\top} = DD^{\top} = X,$$
  which yields
  $$M_1
    = (\hat{P}_1^{\top} A (\hat{P}_1 \hat{P}_1^{\top} X) \hat{P}_1) (\hat{P}_1^{\top} X \hat{P}_1)^{-1}
    = (\hat{P}_1^{\top} A \hat{P}_1) (\hat{P}_1^{\top} X \hat{P}_1) (\hat{P}_1^{\top} X \hat{P}_1)^{-1}
    = \hat{P}_1^{\top} A \hat{P}_1.$$
  This completes the proof of Lemma \ref{thm:least-squares}.
\end{proof}

It might help understanding to note that, when $\hat{P}_1 = P_1$, for any $x_t, x_{t+1} \in E_{\mathrm{u}}$ we have
  $$P_1^{\top} A x_t = y_{t+1} = M_1 y_t = M_1 P_1^{\top} x_t,$$
which requires $P_1^{\top} A = M_1 P_1^{\top}$, or equivalently $M_1 = P_1^{\top} A P_1$ (recall $P_1^{\top} P_1 = I$).
  \section{Transformation of \titlemath{B} with Arbitrary Columns}\label{sec:appendix-transform-k-columns}

In the remaining sections of this paper, we have always regarded $B$ as an $n$-by-$k$ matrix (i.e., $m = k$). In this section, we will show that other cases can be handled in a similar way under proper transformations. This is trivial for the case where $m > k$, since we can simply select $k$ linearly independent columns from $B$, and pad 0's in $u_t$ for all unselected entries.

For the case where $m < k$, let $d = \ceiling{k / m}$. Intuitively, we can ``pack'' every $d$ consecutive steps to obtain a system with sufficient number of control inputs. More specifically, let
\begin{gather*}
  \tilde{x}_{t} = \begin{bmatrix}
    x_{td} \\ x_{td+1} \\ \vdots \\ x_{(t+1)d-1}
  \end{bmatrix},~
  \tilde{u}_{t} = \begin{bmatrix}
    u_{td-1} \\ u_{td} \\ \vdots \\ u_{(t+1)d-2}
  \end{bmatrix},\\
  \tilde{A} = \begin{bmatrix}
    O & & & A \\
    & \ddots & & \vdots \\
    & & O & A^{d-1} \\
    & & & A^d
  \end{bmatrix},~
  \tilde{B} = \begin{bmatrix}
    B \\
    AB & B \\
    \vdots & \vdots & \ddots \\
    A^{d-1}B & A^{d-2}B & \cdots & B
  \end{bmatrix},
\end{gather*}
and consider the transformed system with dynamics
$$\tilde{x}_{t+1} = \tilde{A} \tilde{x}_{t} + \tilde{B} \tilde{u}_{t}.$$
The instability index of $\tilde{A}$ is still $k$, with $|\tilde{\lambda}_i| = |\lambda_i|^d$ ($i=1,\cdots,n$). Norms of $\tilde{A}$ and $\tilde{B}$ satisfy
\begin{equation*}
  \snorm{\tilde{A}} \leq \sqrt{\sum_{i=1}^{d} \snorm{A^i}^2} = \snorm{A^d} O(d),\quad
  \snorm{\tilde{B}} \leq \snorm{B} \sqrt{\sum_{i=1}^{d} (d-i) \snorm{A^i}^2} = \snorm{A^d} \snorm{B} O(d).
\end{equation*}
Since $d \leq k \ll n$, the above transformation only multiplies the bounds by a small constant.
  \section{Proof of Lemma \ref{thm:block-estimate-radius}}\label{sec:appendix-proof-block-estimate-radius}

Lemma \ref{thm:block-estimate-radius} is actually a direct corollary of the following lemma, for which we first need to define $\mathrm{gap}_i(A)$, the \textit{(bipartite) spectral gap around $\lambda_i$ with respect to $A$}, namely
$$\mathrm{gap}_i(A) := \begin{cases}
  \min_{\lambda_j \in \lambda(A_2)} |\lambda_i - \lambda_j| & \lambda_i \in \lambda(A_1) \\
  \min_{\lambda_j \in \lambda(A_1)} |\lambda_i - \lambda_j| & \lambda_i \in \lambda(A_2)
\end{cases},$$
where $\lambda(A)$ denotes the spectrum of $A$.

\begin{lemma}\label{thm:lemma-block-perturb-spectrum}
  For 2-by-2 block matrices $A$ and $E$ in the form
  $$A = \begin{bmatrix}
    A_1 & O \\
    O & A_2
  \end{bmatrix},~
  E = \begin{bmatrix}
    O & E_{12} \\
    E_{21} & O
  \end{bmatrix},$$
  we have
  $$|\lambda_i(A+E) - \lambda_i(A)|
    \leq \frac{\kappa(A) \kappa(A+E)}{\mathrm{gap}_i(A)} \snorm{E_{12}} \snorm{E_{21}}.$$
  Here $\kappa(A)$ is the condition number of the matrix consisting of $A$'s eigenvectors as columns.
\end{lemma}

\begin{proof}
  The proof of the lemma can be found in existing literature like \cite{Naka2017Perturbation}.
\end{proof}

\begin{proof}[Proof of Lemma \ref{thm:block-estimate-radius}]
  Lemma \ref{thm:lemma-block-perturb-spectrum} basically guarantees that every eigenvalue of $A+E$ is within a distance of $O(\snorm{E_{12}} \snorm{E_{21}})$ from some eigenvalue of $A$. Hence, by defining $\chi(A+E)$ as the maximum coefficient, namely
  $$\chi(A+E) := \frac{\kappa(A) \kappa(A+E)}{\min_{i} \sbraced{\mathrm{gap}_i(A)}},$$
  we shall guarantee $|\rho(A+E) - \rho(A)| \leq \chi(A+E) \snorm{E_{12}} \snorm{E_{21}}$.
\end{proof}
  \section{Proof of Theorem \ref{thm:projector-error} and its Corollary}\label{sec:appendix-proof-projector-error}

Without loss of generality, we shall write all matrices in the basis formed by unit eigenvectors $\sbraced{w_1, \cdots, w_n}$ of $A$. Otherwise, let $W = [w_1 ~ \cdots ~ w_n]$, and perform change-of-coordinate by setting $\tilde{D} := W^{-1} D W$, $\tilde{\varPi}_1 := W^{-1} \varPi_1 W$, which further gives
$$\tilde{\hat{\varPi}}_1
  = \tilde{D} (\tilde{D}^{\top} \tilde{D})^{-1} \tilde{D}^{\top}
  = (W^{-1} D W) (W^{-1} D^{\top} D W)^{-1} (W^{-1} D^{\top} W)
  = W^{-1} \hat{\varPi}_1 W.$$
Note that $\snorm{W^{-1} \hat{\varPi}_1 W - W^{-1} \varPi_1 W} \leq \snorm{W} \snorm{W^{-1}} \snorm{\hat{\varPi}_1 - \varPi_1}$, where the upper bound is only magnified by a constant factor of $\cond(W) = \snorm{W} \snorm{W^{-1}}$ that is completely determined by $A$. Therefore, it is largely equivalent to consider $(\tilde{D}, \tilde{\varPi}_1, \tilde{\hat{\varPi}}_1)$ instead of $(D, \varPi_1, \hat{\varPi}_1)$.

Note that the matrix $D = [x_{t_0+1} ~ \cdots ~ x_{t_0+k}]$ can be written as
$$D = \begin{bmatrix}
  d_1 & \lambda_1 d_1 & \cdots & \lambda_1^{k-1} d_1 \\
  d_2 & \lambda_2 d_2 & \cdots & \lambda_2^{k-1} d_2 \\
  \vdots & \vdots & \ddots & \vdots \\
  d_n & \lambda_n d_n & \cdots & \lambda_n^{k-1} d_n
\end{bmatrix},$$
where $x_{t_0+1} =: [d_1, \cdots, d_n]^{\top}$. We first present a lemma characterizing some well-known properties of Vandermonde matrices that we need in the proof.

\begin{lemma}\label{thm:lemma-Vandermonde-matrix}
  Given a Vandermonde matrix in variables $x_1,\cdots,x_n$ of order $n$
  $$V := V_n(x_1,\cdots,x_n) = \begin{bmatrix}
    1 & 1 & \cdots & 1 \\
    x_1 & x_2 & \cdots & x_n \\
    \vdots & \vdots & \ddots & \vdots \\
    x_1^{n-1} & x_2^{n-1} & \cdots & x_n^{n-1}
  \end{bmatrix},$$
  its determinant is given by
  \begin{equation}\label{eq:lemma-Vandermonde-matrix:1}
    \det(V)
    = \sum_{\pi} (-1)^{\sgn(\pi)} x_{\pi(i_1)}^{0} x_{\pi(i_2)}^{1} \cdots x_{\pi(i_n)}^{n-1}
    = \prod_{j<\ell} (x_{\ell}-x_j),
  \end{equation}
  and its $(u,v)$-cofactor is given by
  \begin{equation}\label{eq:lemma-Vandermonde-matrix:2}
    \cof_{u,v}(V) = \begin{vmatrix}
    1 & \cdots & 1 & 1 & \cdots & 1 \\
    \vdots & \ddots & \vdots & \vdots & \ddots & \vdots \\
    x_1^{u-2} & \cdots & x_{v-1}^{u-2} & x_{v+1}^{u-2} & \cdots & x_n^{u-2} \\
    x_1^{u} & \cdots & x_{v-1}^{u} & x_{v+1}^{u} & \cdots & x_n^{u} \\
    \vdots & \ddots & \vdots & \vdots & \ddots & \vdots \\
    x_1^{n-1} & \cdots & x_{v-1}^{n-1} & x_{v+1}^{n-1} & \cdots & x_n^{n-1}
  \end{vmatrix}
  = \sigma_{u,v} \prod_{j<\ell \neq v} (x_{\ell}-x_j),
  \end{equation}
  where $\sigma_{u,v} := s_{n-u}(x_1, \cdots, x_{v-1}, x_{v+1}, \cdots, x_n)$, and $s_m(y_1,\cdots,y_n) := \sum_{i_1 < \cdots < i_m} y_{i_1} \cdots y_{i_m}$.
\end{lemma}

\begin{proof}[Proof of Lemma \ref{thm:lemma-Vandermonde-matrix}]
  The proof of (\ref{eq:lemma-Vandermonde-matrix:1}) can be found in any standard linear algebra textbook, and that of (\ref{eq:lemma-Vandermonde-matrix:2}) can be found in \cite{rawashdeh2019vandermonde}.
\end{proof}

It is evident that the entries in $D$ display a similar pattern as those of a Vandermonde matrix. Based on this observation, we shall further derive the explicit form of $\hat{\varPi}_1$ as in the next lemma.

\begin{lemma}\label{thm:lemma-projector-explicit}
  The projector $\hat{\varPi}_1 = D (D^{\top} D)^{-1} D^{\top}$ has explicit form
  $$(\hat{\varPi}_1)_{uv} = \frac{\disp\sum_{\begin{subarray}{c} i_2 < \cdots < i_k \\ \forall j: i_j \neq u,v \end{subarray}} \alpha_{u, i_2, \cdots, i_k} \alpha_{v, i_2, \cdots, i_k}}{\disp\sum_{i_1 < \cdots < i_k} \alpha_{i_1, \cdots, i_k}^2},$$
  where the summand $\alpha_{i_1, \cdots, i_k}$ (with \textit{ordered} subscript) is defined as
  $$\alpha_{i_1, \cdots, i_k} := \prod_{j} d_{i_j} \prod_{j<\ell} (\lambda_{i_{\ell}} - \lambda_{i_j}).$$
\end{lemma}

\begin{proof}[Proof of Lemma \ref{thm:lemma-projector-explicit}]
We start by deriving the explicit form of $(D^{\top} D)^{-1}$. Note that the determinant (which is also the denominator in the lemma) is given by
\begin{align*}
  \det(D^{\top} D)
  &= \sum_{i_1,\cdots,i_k}
  \begin{vmatrix}
    \lambda_{i_1}^0 d_{i_1}^2 & \lambda_{i_2}^1 d_{i_2}^2 & \cdots & \lambda_{i_k}^{k-1} d_{i_k}^2 \\
    \lambda_{i_1}^1 d_{i_1}^2 & \lambda_{i_2}^2 d_{i_2}^2 & \cdots & \lambda_{i_k}^{k} d_{i_k}^2 \\
    \vdots & \vdots & \ddots & \vdots \\
    \lambda_{i_1}^{k-1} d_{i_1}^2 & \lambda_{i_2}^k d_{i_2}^2 & \cdots & \lambda_{i_k}^{2k-2} d_{i_k}^2
  \end{vmatrix} \\
  &= \sum_{i_1,\cdots,i_k} d_{i_1}^2 \cdots d_{i_k}^2 \lambda_{i_1}^0 \lambda_{i_2}^1 \cdots \lambda_{i_k}^{k-1} \prod_{j<\ell} (\lambda_{i_{\ell}} - \lambda_{i_j}) \\
  &= \sum_{i_1<\cdots<i_k} d_{i_1}^2 \cdots d_{i_k}^2 \prod_{j<\ell} (\lambda_{i_{\ell}} - \lambda_{i_j}) \sum_{\pi} (-1)^{\sgn(\pi)} \lambda_{\pi(j_1)}^{0} \lambda_{\pi(j_2)}^{1} \cdots \lambda_{\pi(j_k)}^{k-1} \\
  &= \sum_{i_1<\cdots<i_k} d_{i_1}^2 \cdots d_{i_k}^2 \prod_{j<\ell} (\lambda_{i_{\ell}} - \lambda_{i_j})^2 \\
  &= \sum_{i_1<\cdots<i_k} \alpha_{i_1, \cdots, i_k}^2,
\end{align*}
and the $(u,v)$-cofactor $\cof_{u,v}(D^{\top} D)$ is given by
\begin{align*}
  \cof_{u,v}(D^{\top} D)
  &= (-1)^{u+v} \sum_{i_1,\cdots,i_{k-1}}
  \begin{vmatrix}
    \lambda_{i_1}^0 d_{i_1}^2 & \cdots & \lambda_{i_{v-1}}^{v-2} d_{i_{v-1}}^2 & \lambda_{i_{v}}^{v} d_{i_{v}}^2 & \cdots & \lambda_{i_{k-1}}^{k-1} d_{i_{k-1}}^2 \\
    \vdots & \ddots & \vdots &\vdots & \ddots & \vdots \\
    \lambda_{i_1}^{u-2} d_{i_1}^2 & \cdots & \lambda_{i_{v-1}}^{u+v-4} d_{i_{v-1}}^2 & \lambda_{i_{v}}^{u+v-2} d_{i_{v}}^2 & \cdots & \lambda_{i_{k-1}}^{u+k-3} d_{i_{k-1}}^2 \\
    \lambda_{i_1}^{u} d_{i_1}^2 & \cdots & \lambda_{i_{v-1}}^{u+v-2} d_{i_{v-1}}^2 & \lambda_{i_{v}}^{u+v} d_{i_{v}}^2 & \cdots & \lambda_{i_{k-1}}^{u+k-1} d_{i_{k-1}}^2 \\
    \vdots & \ddots & \vdots &\vdots & \ddots & \vdots \\
    \lambda_{i_1}^{k-1} d_{i_1}^2 & \cdots & \lambda_{i_{u+v-2}}^{k+v-3} d_{i_{v-1}}^2 & \lambda_{i_{v}}^{k+v-1} d_{i_{v}}^2 & \cdots & \lambda_{i_{k-1}}^{2k-2} d_{i_{k-1}}^2
  \end{vmatrix} \\
  &= (-1)^{u+v} \sum_{i_1,\cdots,i_{k-1}} d_{i_1}^2 \cdots d_{i_{k-1}}^2 \lambda_{i_1}^0 \cdots \lambda_{i_{v-1}}^{v-2} \lambda_{i_v}^{v} \cdots \lambda_{i_{k-1}}^{k-1} s_{k-u} \prod_{j<\ell} (\lambda_{i_{\ell}} - \lambda_{i_j}) \\
  &= (-1)^{u+v} \sum_{i_1<\cdots<i_{k-1}} s_{k-u} \cdot d_{i_1}^2 \cdots d_{i_{k-1}}^2 \prod_{j<\ell} (\lambda_{i_{\ell}} - \lambda_{i_j}) \cdot\\
  &\phantom{{}= (-1)^{u+v} \sum_{i_1<\cdots<i_{k-1}}{}} \sum_{\pi} (-1)^{\sgn(\pi)} \lambda_{\pi(i_1)}^{0}  \cdots \lambda_{\pi(i_{v-1})}^{v-2} \lambda_{\pi(i_v)}^{v} \cdots \lambda_{\pi(i_{k-1})}^{k-1} \\
  &= (-1)^{u+v} \sum_{i_1<\cdots<i_{k-1}} s_{k-u} s_{k-v} \cdot d_{i_1}^2 \cdots d_{i_{k-1}}^2 \prod_{j<\ell} (\lambda_{i_{\ell}} - \lambda_{i_j})^2,
\end{align*}
where $s_{k-u}(\lambda_{i_1},\cdots,\lambda_{i_{k-1}})$ is abbreviated to $s_{k-u}$.

Note that symmetry of $D^{\top} D$ guarantees $\cof_{v,u}(D^{\top} D) = \cof_{u,v}(D^{\top} D)$, so we have
$$(D^{\top} D)^{-1}_{u,v} = \frac{\cof_{v,u}(D^{\top} D)}{\det(D^{\top} D)} = \frac{\cof_{u,v}(D^{\top} D)}{\det(D^{\top} D)}.$$
And eventually we shall derive that
\begin{align*}
  \hat{P}_{u,v}
  &= \sum_{p,q} D_{u,p} (D^{\top} D)^{-1}_{p,q} D^{\top}_{q,v} \\
  &= \frac{1}{\det(D^{\top} D)} \sum_{p,q} D_{u,p} D_{v,q} \cof_{u,v}(D^{\top} D) \\
  &= \frac{1}{\det(D^{\top} D)} \sum_{p,q} \lambda_u^{p-1} d_u \lambda_v^{q-1} d_v \cdot (-1)^{p+q} \sum_{i_1<\cdots<i_{k-1}} s_{k-p} s_{k-q} \cdot d_{i_1}^2 \cdots d_{i_{k-1}}^2 \prod_{j<\ell} (\lambda_{i_{\ell}} - \lambda_{i_j})^2 \\
  &= \frac{1}{\det(D^{\top} D)} \sum_{i_1<\cdots<i_{k-1}} d_u d_v d_{i_1}^2 \cdots d_{i_{k-1}}^2 \prod_{j<\ell} (\lambda_{i_{\ell}} - \lambda_{i_j})^2 \sum_{p=1}^{k} (-1)^{p} \lambda_u^{p-1} s_{k-p} \sum_{q=1}^{k} (-1)^{q} \lambda_v^{q-1} s_{k-q} \\
  &= \frac{1}{\det(D^{\top} D)} \sum_{i_1<\cdots<i_{k-1}} d_u d_{i_1} \cdots d_{i_{k-1}} \prod_{j<\ell} (\lambda_{i_{\ell}} - \lambda_{i_j}) \prod_{\ell} (\lambda_{i_\ell}-\lambda_u) \cdot \\
  & \phantom{= \frac{1}{\det(D^{\top} D)} \sum_{i_1<\cdots<i_{k-1}}}~ d_v d_{i_1} \cdots d_{i_{k-1}} \prod_{j<\ell} (\lambda_{i_{\ell}} - \lambda_{i_j}) \prod_{\ell} (\lambda_{i_\ell}-\lambda_v) \\
  &= \frac{1}{\det(D^{\top} D)} \disp\sum_{\begin{subarray}{c} i_2 < \cdots < i_k \\ \forall j: i_j \neq u,v \end{subarray}} \alpha_{u, i_2, \cdots, i_k} \alpha_{v, i_2, \cdots, i_k},
\end{align*}
which is in exact the same form as stated in the lemma.
\end{proof}

Now we shall go back to the proof of the main result of this section.  

\begin{proof}[Proof of Theorem \ref{thm:projector-error}]
Recall that $d_i = \lambda_i^{t_0+1} x_{0,i}$. For the clarity of notations, let
$$\theta_{i_1, i_2, \cdots, i_k} := \frac{\alpha_{i_1, i_2, \cdots, i_k}}{\alpha_{1,2,\cdots,k}},$$
and it is evident that $|\theta_{i_1, i_2, \cdots, i_k}| = 1$ only if $(i_1, i_2, \cdots, i_k)$ is a permutation of $(1,2,\cdots,k)$. For any other $(i_1, i_2, \cdots, i_k)$, by the definition in Lemma \ref{thm:lemma-projector-explicit} we have
$$|\theta_{i_1, i_2, \cdots, i_k}|
  \leq c_{i_1, i_2, \cdots, i_k} \cdot r^{\delta(i_1, i_2, \cdots, i_k) t_0}
  \leq c \cdot r^{\delta(i_1, i_2, \cdots, i_k) t_0},$$
where $r = \max\limits_{i} \sbraced{\frac{|\lambda_{i+1}|}{|\lambda_i|}}$, $c := \max\limits_{i_1,\cdots,i_k} \sbraced{c_{i_1, i_2, \cdots, i_k}}$, and $\delta(i_1, i_2, \cdots, i_k) := \sum_{j} i_j - \frac{k(k+1)}{2} \in \N$. Therefore, $|\theta_{i_1, i_2, \cdots, i_k}|$ will be small when $(1,2,\cdots,k)$ is ``far away'' from $(i_1, i_2, \cdots, i_k)$.

To get a tighter bound, we need to analyze the distribution of $\delta(\cdot)$ in the exponent. For any fixed $\delta = \delta(i_1, i_2, \cdots, i_k)$, there are $q(\delta+\tfrac{k(k+1)}{2},k)$ different tuples, where $q(n,k)$ denotes the number of different methods to partition $n$ into $k$ distinct integer parts. Then we have
$$\sum_{i_1 < \cdots < i_k} \theta_{i_1, \cdots, i_k}^2 - \theta_{1, \cdots, k}^2
  = c \sum_{\delta=0}^{k(n-k)} q(\delta+\tfrac{k(k+1)}{2}, k) r^{2\delta t_0}
  \leq c \cdot Q_k(r^{2t_0}) r^{-k(k+1)t_0},$$
where $Q_k(x) := \sum_{n} q(n,k) x^n$ is the generating function for $q(n,k)$ with fixed $k$, which is
$$Q_k(x) = x^{k(k+1)/2} \prod_{j=1}^{k} \frac{1}{1-x^j},$$
Hence we conclude that
$$\sum_{i_1 < \cdots < i_k} \theta_{i_1, \cdots, i_k}^2 - \theta_{1, \cdots, k}^2
  \leq c \paren{\prod_{j=1}^{k} \frac{1}{1-r^{2jt_0}} - 1},$$
which monotone-increasingly converges to a constant $c\gamma(r, t_0) = \frac{c}{(r^{2t_0};r^{2t_0})_{\infty}} - c$ as $k \to \infty$, where $(\cdot; \cdot)_{\infty}$ is the \textit{q-Pochhammer symbol}. Note that
$$(x; x)_{\infty} = 1 - x + O(r^{4t_0}) ~\Rightarrow~
  \gamma(r, t_0) = r^{2t_0} + O(r^{4t_0}),$$
we know that $\gamma(r, t_0) \leq 2r^{2t_0}$ when $r^{t_0}$ is sufficiently small. For the nominator, note that for each $\delta$ there are fewer entries with exponent $\delta$ in the nominator than in the denominator, so we also have
$$\abs{\sum_{\begin{subarray}{c} i_2 < \cdots < i_k \\ \forall j: i_j \neq u,v \end{subarray}} \theta_{u, i_2, \cdots, i_k} \theta_{v, i_2, \cdots, i_k}} 
\leq \begin{cases}
  c\gamma(r, t_0) + 1 & u = v \leq k \\
  c\gamma(r, t_0) & \textrm{otherwise}
\end{cases}.$$

Eventually, for any $\varepsilon > 0$, we shall select $t_0$ such that $c\gamma(r, t_0) < \frac{\varepsilon}{n^2}$, where the denominator is always bounded by
$$1 \leq \sum_{i_1 < \cdots < i_k} \theta_{i_1, \cdots, i_k}^2 \leq 1 + \frac{\varepsilon}{n^2}.$$
For the nominator, when $u = v \leq k$, we have $\disp\sum_{\begin{subarray}{c} i_2 < \cdots < i_k \\ \forall j: i_j \neq u \end{subarray}} \theta_{u, i_2, \cdots, i_k}^2 \geq 1$, which shows
$$\left. \begin{aligned}
  (\hat{\varPi}_1)_{uv} \geq \paren{1 + \frac{\varepsilon}{n^2}}^{-1} \geq 1 - \frac{\varepsilon}{n^2}& \\
  (\hat{\varPi}_1)_{uv} \leq 1 + \frac{\varepsilon}{n^2}&
\end{aligned} \right\rbrace  
~\Rightarrow~ \abs{(\hat{\varPi}_1)_{uv} - (\varPi_1)_{uv}} \leq \frac{\varepsilon}{n^2}.$$
Otherwise, the nominator cannot sum over a permutation of $(1, \cdots, k)$, which gives
$$\abs{(\hat{\varPi}_1)_{uv} - (\varPi_1)_{uv}} = \abs{(\hat{\varPi}_1)_{uv}} \leq \frac{\varepsilon}{n^2}.$$
Therefore, the overall estimation error is bounded by
$$\snorm{\hat{\varPi}_1 - \varPi_1} \leq \sum_{u,v} \abs{(\hat{\varPi}_1)_{uv} - (\varPi_1)_{uv}} \leq \varepsilon.$$
To achieve error threshold $\varepsilon$, it is required that $2cr^{2t_0} < \frac{\varepsilon}{n^2}$, or equivalently
$$t_0 = O \paren{\frac{\log \frac{n}{\varepsilon}}{\log \frac{1}{r}}}.$$
This completes the proof.
\end{proof}

\begin{proof}[Proof of Corollary \ref{thm:projector-error-corollary}]
  We first construct a specific pair of orthonormal bases $(P_1^*, \hat{P}_1^*)$ that satisfy the corollary. To start with, take an arbitrary initial pair of orthonormal basis $(P_1^{\circ}, \hat{P}_1^{\circ})$, and consider the SVD $(P_1^{\circ})^{\top} \hat{P}_1^{\circ} = U \varSigma V^{\top}$, which is equivalent to $(P_1^{\circ} U)^{\top} (\hat{P}_1^{\circ} V) = \varSigma$. Note that the columns of $P_1^{\circ} U = [w_1 ~ \cdots w_k]$ and $\hat{P}_1^{\circ} V = [\hat{w}_1 ~ \cdots \hat{w}_k]$ form orthonormal bases of $\col(\varPi_1)$ and $\col(\hat{\varPi}_1)$, respectively; furthermore, these bases project onto each other accordingly by subscripts, namely
  $$\varPi_1 \hat{w}_i = \sigma_i w_i,~\hat{\varPi_1} w_i = \sigma_i \hat{w}_i.$$
  Now we set $P_1^* := P_1^{\circ} U$ and $\hat{P}_1^* := \hat{P}_1^{\circ} V$. Note that
  $$|1 - \sigma_i| = \snorm{(\hat{\varPi}_1 - \varPi_1) \hat{w}_i} < \varepsilon,$$
  which shows, by properties of projection matrix $\varPi_1$,
  $$\snorm{w_i - \hat{w}_i}
    = \sqrt{\snorm{w_i - \varPi_1 \hat{w}_i}^2 + \snorm{\varPi_1 \hat{w}_i - \hat{w}_i}^2}
    = \sqrt{|1 - \sigma_i|^2 + \snorm{(\hat{\varPi}_1 - \varPi_1) \hat{w}_i}^2}
    < \sqrt{2} \varepsilon,$$
  and thus
  $$\snorm{P_1^* - \hat{P}_1^*}
    = \max_{\snorm{z}=1} \snorm{(P_1^*-\hat{P}_1^*) z}
    = \max_{\snorm{z}=1} \norm{\sum_{i} z_i (w_i-\hat{w}_i)}
    \leq \sqrt{k} \cdot \sqrt{2}\varepsilon.$$
  To further generalize the proposition to any arbitrary $\hat{P}_1$, we only have to note that there exists an orthonormal matrix $T$ that maps the basis $\hat{P}_1^*$ to $\hat{P}_1 = \hat{P}_1^* T$. Now take $P_1 = P_1^* T$, and we have
  $$\snorm{\hat{P}_1 - P_1} = \snorm{(\hat{P}_1^* - P_1^*) T} = \snorm{\hat{P}_1^* - P_1^*} < \sqrt{2k} \varepsilon.$$
  As for the estimation error bound for $M_1$, we can directly write
  \begin{align*}
    \snorm{P_1^{\top} A P_1 - \hat{P}_1^{\top} A \hat{P}_1}
    &\leq \snorm{P_1^{\top} A P_1 - P_1^{\top} A \hat{P}_1} + \snorm{P_1^{\top} A \hat{P}_1 - \hat{P}_1^{\top} A \hat{P}_1} \\
    &\leq \snorm{A} \snorm{P_1 - \hat{P}_1} + \snorm{A} \snorm{P_1 - \hat{P}_1} \\
    &< 2 \snorm{A} \delta,
  \end{align*}
  This completes the proof of the corollary.
\end{proof}

Recall that we are allowed to take any orthonormal basis $P_1$ for $E_{\mathrm{u}}$. Hence we shall always assume by default that $P_1$ in the proofs are selected as shown in the proof above.

We finish this section with simple but frequently-used bounds on $\snorm{\hat{P}_1^{\top} P_1}$ and $\snorm{\hat{P}_1^{\top} P_2}$. These factors represent an additional error introduced by using the inaccurate projector $\hat{P}_1$.

\begin{proposition}\label{prop:norm-bound-projection-error}
  Under the premises of Corollary \ref{thm:projector-error-corollary}, $\snorm{I_k - \hat{P}_1^{\top} P_1} < \delta$, $\snorm{\hat{P}_1^{\top} P_2} < \delta$.
\end{proposition} 

\begin{proof}
  Note that $P_1^{\top} P_1 = I_k$ and $P_1^{\top} P_2 = O$, it is evident that
  \begin{gather*}
    \snorm{I_k - \hat{P}_1^{\top} P_1} = \snorm{(P_1 - \hat{P}_1)^{\top} P_1} < \delta,\\
    \snorm{\hat{P}_{1}^{\top} P_2} = \snorm{(\hat{P}_1 - P_1)^{\top} P_2} = \snorm{\hat{P}_1 - P_1} < \delta.
  \end{gather*}
  This finishes the proof.
\end{proof}

  \section{Proof of Theorem \ref{thm:main-theorem-symmetric-A}}\label{sec:appendix-proof-symmetric-theorem}

We start by showing the estimation error bound for $B_1$, which is straight-forward since $\varDelta = O$. Note that the upper bound of the norm of our controller $\hat{K}_1$ appears as a natural corollary of it.

\begin{proposition}\label{thm:estimation-error-B-symmetric-A}
  Under the premises of Theorem \ref{thm:main-theorem-symmetric-A}, $\snorm{\hat{B}_1 - B_1} < 4 \snorm{A} \sqrt{k} \delta$.
\end{proposition}

\begin{proof}
  Note that the column vector $b_i$ has estimation error bound
  \begin{align*}
    \snorm{b_i - \hat{b}_i}
    &= \frac{1}{\snorm{x_{t_i}}} \norm{\big( P_1^{\top} x_{t_i+1} - M_1 P_1^{\top} x_{t_i} \big) - \big( \hat{P}_1^{\top} x_{t_i+1} - \hat{M}_1 \hat{P}_1^{\top} x_{t_i} \big)} \\
    &\leq \frac{1}{\snorm{x_{t_i}}} \paren{\snorm{(P_1^{\top} - \hat{P}_1^{\top}) A x_{t_i}} + \snorm{(M_1 P_1^{\top} - \hat{M}_1 \hat{P}_1^{\top}) x_{t_i}}} \\
    &\leq \snorm{P_1^{\top} - \hat{P}_1^{\top}} \snorm{A} + \snorm{M_1 P_1^{\top} - M_1 \hat{P}_1^{\top}} + \snorm{M_1 \hat{P}_1^{\top} - \hat{M}_1 \hat{P}_1^{\top}} \\
    &< \snorm{A}\delta + \snorm{M_1} \snorm{P_1^{\top}-\hat{P}_1^{\top}} + \snorm{M_1 - \hat{M}_1} \\
    &< \snorm{A}\delta + \snorm{A}\delta + 2\snorm{A}\delta
    = 4\snorm{A}\delta,
  \end{align*}
  where we repeatedly apply Corollary \ref{thm:projector-error-corollary} and the fact that $\snorm{M_1} \leq \snorm{A}$. Then, to bound the error of the whole matrix, we simply apply the definition
  $$\snorm{\hat{B}_1 - B_1}
    = \max_{\snorm{u}=1} \snorm{(\hat{B}_1 - B_1)u}
    \leq \max_{\snorm{u}=1} \sum_{i=1}^{k} |u_i| \snorm{\hat{b}_i - b_i}
    < 4 \snorm{A} \sqrt{k} \delta.$$
  This completes the proof.
\end{proof}

\begin{corollary}\label{thm:norm-bound-K1-symmetric-A}
  Under the premises of Theorem \ref{thm:main-theorem-symmetric-A}, when (\ref{eq:constraint-delta-symmetric-A:1}) holds, $\snorm{\hat{K}_1} < \frac{2\snorm{A}}{c\snorm{B}}$.
\end{corollary}

\begin{proof}
  By Proposition \ref{thm:estimation-error-B-symmetric-A}, it is evident that
  $$\sigma_{\min}(\hat{B}_1)
    \geq \sigma_{\min}(B_1) - \snorm{\hat{B}_1 - B_1}
    > (c - 4\snorm{A}\sqrt{k} \delta) \snorm{B}
    > \frac{c}{2} \snorm{B},$$
  where the last inequality requires
  \begin{equation}\label{eq:constraint-delta-symmetric-A:1}
    \delta < \frac{c}{8\snorm{A}\sqrt{k}}.
  \end{equation}
  Recall that $\hat{K}_1 = \hat{B}_1^{-1} \hat{M}_1$, and note that $\snorm{\hat{B}_1^{-1}} \leq \frac{1}{\sigma_{\min}(\hat{B}_1)}$, so we have
  $$\snorm{\hat{K}_1}
    = \snorm{\hat{B}_1^{-1} \hat{M}_1}
    \leq \frac{\snorm{\hat{P}_1^{\top} A \hat{P}_1}}{\sigma_{\min}(\hat{B}_1)}
    < \frac{2\snorm{A}}{c\snorm{B}}.$$
  This completes the proof.
\end{proof}

Recall that to apply Lemma \ref{thm:block-estimate-radius}, we need a bound on the spectral radii of diagonal blocks. The top-left block has already been eliminated to approximately $O$ by the design of $\hat{K}_1$, but the bottom-right block needs some extra work --- although $M_2$ is known to be stable, the inaccurate projection introduces an extra error that perturbs the spectrum. To bound the perturbed spectral radius, we will apply the following perturbation bound known as Bauer-Fike Theorem.

\begin{lemma}[Bauer-Fike]\label{thm:perturbation-Elsner}
  Suppose $A \in \R^{n \times n}$ is diagonalizable, then for any $E \in \R^{n \times n}$, we have
  $$|\rho(A) - \rho(A+E)| \leq \max_{\hat{\lambda} \in \lambda(A+E)} \min_{\lambda \in \lambda(A)} |\lambda - \hat{\lambda}| \leq \kappa(A) \snorm{E},$$
  where $\kappa(A)$ is the condition number of the matrix consisting of $A$'s eigenvectors as columns (i.e., if $A = S \varLambda S^{-1}$ with diagonal $\varLambda$, then $\kappa(A) = \cond(S)$), and $\lambda(A)$ denotes the spectrum of $A$.
\end{lemma}

\begin{proof}
  The proof is well-known and can be found in, e.g., \cite{Bauer1960perturbation}.
\end{proof}

Now we are ready to prove the main theorem for any symmetric dynamical matrix $A$.

\begin{proof}[Proof of Theorem \ref{thm:main-theorem-symmetric-A}]
  With $\tau = 1$, the controlled dynamics under estimated controller $\hat{K}_1$ becomes
  $$\hat{L}_{1} = \begin{bmatrix}
    M_1 + P_1^{\top} B \hat{K}_1 \hat{P}_1^{\top} P_1
      & P_1^{\top} B \hat{K}_1 \hat{P}_1^{\top} P_2 \\
    P_2^{\top} B \hat{K}_1 \hat{P}_1^{\top} P_1
      & M_2 + P_2^{\top} B \hat{K}_1 \hat{P}_1^{\top} P_2
  \end{bmatrix}.$$
  
  We first guarantee that the diagonal blocks are stable. For the top-left block,
  \begin{align}
    \snorm{M_1 + P_1^{\top} B\hat{K}_1}
    &= \snorm{M_1 - B_1 \hat{B}_1^{-1} \hat{M}_1 \hat{P}_1^{\top} P_1} \nonumber\\
    &\leq \snorm{M_1 - \hat{M}_1} + \snorm{\hat{M}_1 - B_1 \hat{B}_1^{-1} \hat{M}_1} + \snorm{B_1 \hat{B}_1^{-1} \hat{M}_1 (I_k - \hat{P}_1^{\top} P_1)} \nonumber\\
    &\leq \snorm{M_1 - \hat{M}_1} + \snorm{\hat{B}_1 - B_1} \snorm{\hat{K}_1} + \snorm{B} \snorm{\hat{K}_1} \snorm{I_k - \hat{P}_1^{\top} P_1} \nonumber\\
    &< 2\snorm{A}\delta + \frac{8\snorm{A}^2 \sqrt{k}}{c \snorm{B}} \delta + \frac{2 \snorm{A}}{c} \delta \label{eq:main-theorem-symmetric-A:1}\\
    &= \frac{2 \big( 4\sqrt{k}\snorm{A} + (c+1)\snorm{B} \big) \snorm{A}}{c \snorm{B}} \delta \nonumber,
  \end{align}
  where in (\ref{eq:main-theorem-symmetric-A:1}) we apply Corollary \ref{thm:projector-error-corollary},  Corollary \ref{thm:norm-bound-K1-symmetric-A}, and Proposition \ref{prop:norm-bound-projection-error}. Meanwhile, for the bottom-right block, note that the norm of the error term is bounded by
  $$\snorm{P_2^{\top} B \hat{K}_1 \hat{P}_1^{\top} P_2}
    \leq \snorm{B} \snorm{\hat{B}_1^{-1}} \snorm{\hat{M}_1} \snorm{\hat{P}_1^{\top} P_2}
    \leq \frac{2 \snorm{A}}{c} \delta.$$
  Hence, by Lemma \ref{thm:perturbation-Elsner}, the spectral radius of the bottom-right block is bounded by
  $$\rho(M_2 + P_2^{\top} B \hat{K}_1 \hat{P}_1^{\top} P_2)
    \leq \rho(M_2) + \tfrac{2}{c} \kappa(M_2) \snorm{A} \delta
    < 1,$$
  where we require (recall that $\rho(M_2) = |\lambda_{k+1}|$)
  \begin{equation}\label{eq:constraint-delta-symmetric-A:2}
    \delta < \frac{c(1-|\lambda_{k+1}|)}{2 \kappa(M_2) \snorm{A}}.
  \end{equation}
  
  To apply the lemma, it only suffices to bound the spectral norms of off-diagonal blocks. Note that the top-right block is bounded by
  $$\snorm{P_1^{\top} B \hat{K}_1 \hat{P}_1^{\top} P_2}
    \leq \snorm{B} \snorm{\hat{K}_1} \snorm{\hat{P}_1^{\top} P_2}
     < \frac{2 \snorm{A}}{c} \delta,$$
  and the bottom-left block is bounded by
  $$\snorm{P_2^{\top} B \hat{K}_1 \hat{P}_1^{\top} P_1}
    \leq \snorm{B} \snorm{\hat{K}_1}
    \leq \frac{2 \snorm{A}}{c}.$$
  Now, by Lemma \ref{thm:block-estimate-radius}, we can guarantee that
  $$\rho(\hat{L}_1)
    \leq \max\braced{\frac{2 \big( 4\sqrt{k}\snorm{A} + 2(c+1)\snorm{B} \big) \snorm{A}}{c \snorm{B}} \delta, |\lambda_{k+1}| + \snorm{B} \snorm{\hat{K}_1} \delta} + \frac{4\snorm{A}^2 \chi(\hat{L}_1)}{c^2} \delta
    < 1,$$
  where we require
  \begin{equation}\label{eq:constraint-delta-symmetric-A:3}
    \delta < \min\braced{
      \frac{1}{\frac{2 \big( 4\sqrt{k}\snorm{A} + 2(c+1)\snorm{B} \big) \snorm{A}}{c \snorm{B}} + \frac{4\snorm{A}^2 \chi(\hat{L}_1)}{c^2}},
      \frac{1-|\lambda_{k+1}|}{\frac{2\snorm{A}}{c} + \frac{4\snorm{A}^2 \chi(\hat{L}_1)}{c^2}}
    }.
  \end{equation}
  
  So far, it is still left to recollect all the constraints we need on $\delta$ (see (\ref{eq:constraint-delta-symmetric-A:1}), (\ref{eq:constraint-delta-symmetric-A:2}) and (\ref{eq:constraint-delta-symmetric-A:3})), i.e.,
  \begin{equation*}
    \delta < \min\braced{
      \frac{c}{8\snorm{A}\sqrt{k}},
      \frac{c(1-|\lambda_{k+1}|)}{2 \kappa(M_2) \snorm{A}},
      \frac{1-|\lambda_{k+1}|}{\frac{2\snorm{A}}{c} + \frac{4\snorm{A}^2 \chi(\hat{L}_1)}{c^2}}, \frac{1}{\frac{2 \big( 4\sqrt{k}\snorm{A} + 2(c+1)\snorm{B} \big) \snorm{A}}{c \snorm{B}} + \frac{4\snorm{A}^2 \chi(\hat{L}_1)}{c^2}}
    },
  \end{equation*}
  which can be simplified (but weakened) to
  \begin{equation}\label{eq:constant-delta-symmetric-merge}
    \delta
    < \frac{c^2 (1-|\lambda_{k+1}|)}{16 \sqrt{k} \kappa(M_2) \snorm{A} (\snorm{A}+\snorm{B}) \chi(\hat{L}_1)}
    = O(k^{-1/2})
  \end{equation}
  This completes the proof of Theorem \ref{thm:main-theorem-symmetric-A}.
\end{proof}

  \section{Proof of the Main Theorem}\label{sec:appendix-proof-general-theorem}

Technically, we would like to bound the spectral radius of the matrix
$$\hat{L}_{\tau} = \begin{bmatrix}
  M_1^{\tau} + P_1^{\top} A^{\tau-1}B \hat{K}_1 \hat{P}_1^{\top} P_1 & \varDelta_{\tau} + P_1^{\top} A^{\tau-1}B \hat{K}_1 \hat{P}_1^{\top} P_2 \\
  P_2^{\top} A^{\tau-1}B \hat{K}_1 \hat{P}_1^{\top} P_1 & M_2^{\tau} + P_2^{\top} A^{\tau-1}B \hat{K}_1 \hat{P}_1^{\top} P_2.
\end{bmatrix}$$
using Lemma \ref{thm:block-estimate-radius}. The proof is split into two major building blocks: on the one hand, we introduce the well-known Gelfand's Formula to bound matrices appearing with exponents; on the other hand, we establish the estimation error bound for $B_{\tau}$ (parallel to Lemma \ref{thm:estimation-error-B-symmetric-A}) and proceed to bound $\snorm{\hat{K}_1}$, for which we rely on the instability results shown in Section \ref{subsec:instability-unstable}. Finally, a combination of these building blocks naturally establishes the main theorem.

\subsection{Gelfand's Formula}

In this section, we will show norm bounds for factors that contain matrix exponents. It is natural to apply the well-known Gelfand's formula as stated below.

\begin{lemma}[Gelfand's formula]\label{thm:Gelfand-formula}
  For any square matrix $X$, we have
  \begin{equation}\label{eq:Gelfand-formula:1}
    \rho(X) = \lim_{t \to \infty} \snorm{X^t}^{1/t}.
  \end{equation}
  In other words, for any $\varepsilon > 0$, there exists a constant $\zeta_{\varepsilon}(X)$ such that
  \begin{equation}\label{eq:Gelfand-formula:2}
    \sigma_{\max}(X^t) = \snorm{X^t} \leq \zeta_{\varepsilon}(X) (\rho(X)+\varepsilon)^t.
  \end{equation}
  Further, if $X$ is invertible, let $\lambda_{\min}(X)$ denote the eigenvalue of $X$ with minimum modulus, then
  \begin{equation}\label{eq:Gelfand-formula:3}
    \sigma_{\min}(X^t) \geq \frac{1}{\zeta_{\varepsilon}(X^{-1})} \paren{\frac{|\lambda_{\min}(X)|}{1 + \varepsilon |\lambda_{\min}(X)|}}^t.
  \end{equation}
\end{lemma}

\begin{proof}
  The proof of (\ref{eq:Gelfand-formula:1}) can be easily found in existing literature (e.g., \cite{HornMatrixAnalysis}, Corollary 5.6.14), and (\ref{eq:Gelfand-formula:2}) follows by the definition of limits. For (\ref{eq:Gelfand-formula:3}), note that
  $$\sigma_{\min}(X^t)
    = \frac{1}{\sigma_{\max}((X^{-1})^t)}
    \geq \frac{1}{\zeta_{\varepsilon}(X^{-1}) (\rho(X^{-1})+\varepsilon)^t}
    = \frac{1}{\zeta_{\varepsilon}(X^{-1})} \paren{\frac{|\lambda_{\min}(X)|}{1 + \varepsilon |\lambda_{\min}(X)|}}^t,$$
  where we apply $\sigma_{\min}(X^t) = \sigma_{\max}((X^{-1})^t)^{-1}$ and $\rho(X^{-1}) = |\lambda_{\min}(X)|^{-1}$.
\end{proof}

It is evident that $\rho(A) = \rho(M_1) = \rho(N_1) = |\lambda_1|$, $\lambda_{\min}(M_1) = \lambda_{\min}(N_1) = |\lambda_k|$ and $\rho(M_2) = \rho(N_2) = |\lambda_{k+1}|$ (recall that $M_1$ and $M_2$ inherits the unstable and stable eigenvalues, respectively). Therefore, we can use Gelfand's formula to bound the relevant factors appearing in $\hat{L}_{\tau}$.

\begin{proposition}\label{prop:norm-bound-Gelfand}
  Under the premises of Theorem \ref{thm:main-theorem}, the following results hold for any $t \in \N$:
  \begin{enumerate}
    \item $\snorm{B_t} \leq \zeta_{\varepsilon_1}(A) (|\lambda_1| + \varepsilon_1)^{t-1} \snorm{B}$;
    \item $\snorm{P_2^{\top} A^t} \leq \zeta_{\varepsilon_2}(M_2) (|\lambda_{k+1}| + \varepsilon_2)^t$;
    \item $\snorm{\varDelta_t} \leq C_{\varDelta} (|\lambda_1| + \varepsilon_1)^t$, where $C_{\varDelta} = \zeta_{\varepsilon_1}(M_1) \zeta_{\varepsilon_2}(M_2) \frac{(2-\xi)\sqrt{2\xi} \snorm{A}}{1 - \xi} \frac{2 |\lambda_{k+1}|}{|\lambda_1| + \varepsilon_1 - |\lambda_{k+1}| - \varepsilon_2}$. 
  \end{enumerate}
  Here (and below) $\varepsilon_1$ and $\varepsilon_2$ are selected to be sufficiently small constants (see (\ref{eq:constraint-epsilons})).
\end{proposition}

\begin{proof}
  (1) This is a direct corollary of Gelfand's Formula, since
  $$\snorm{B_t}
    = \snorm{P_1^{\top} A^{t-1} B}
    \leq \snorm{A^{t-1}} \snorm{B}
    \leq \zeta_{\varepsilon_1}(A) (|\lambda_1| + \varepsilon_1)^{t-1} \snorm{B}.$$

  (2) It only suffices to recall $\rho(M_2) = |\lambda_{k+1}|$, and note that 
  $$P_2^{\top} A^t = P_2^{\top} P M^t P^{-1} = [O ~ I_{n-k}] M^t P^{\top} = M_2^t P_2^{\top}.$$
  Hence by Gelfand's formula we have $\snorm{P_2^{\top} A^t} = \snorm{M_2^t} \leq \zeta_{\varepsilon_2}(M_2) (|\lambda_{k+1}| + \varepsilon_2)^t$.
  
  (3) This is a direct corollary of Lemma \ref{thm:lemma-xi-close-interpretation}(4) and Gelfand's formula, since
  \begin{align*}
    \norm{\varDelta_t}
    &= \norm{\sum_{i} M_1^{i} \varDelta M_2^{t-1-i}}
    \leq \snorm{\varDelta} \sum_{i} \snorm{M_1^i} \snorm{M_2^{t-1-i}} \\
    &\leq \zeta_{\varepsilon_1}(M_1) \zeta_{\varepsilon_2}(M_2) \frac{(2-\xi)\sqrt{2\xi} \snorm{A}}{1 - \xi} \sum_{i} (\varepsilon_1+|\lambda_{1}|)^{i} (|\lambda_{k+1}| + \varepsilon_2)^{t-1-i} \\
    &= C_{\varDelta} (|\lambda_1| + \varepsilon_1)^t.
  \end{align*}
  This finishes the proof of the proposition.
\end{proof}

\begin{proposition}\label{thm:estimation-error-M1-tau}
  Under the premises of Theorem \ref{thm:main-theorem}, 
  $$\snorm{\hat{M}_1^{\tau} - M_1^{\tau}} < 2 \tau \snorm{A} \zeta_{\varepsilon_1}(A)^2 (|\lambda_1| + \varepsilon_1)^{\tau-1} \delta.$$
\end{proposition}

\begin{proof}
  Recall that Corollary \ref{thm:projector-error-corollary} gives $\snorm{M_1 - \hat{M}_1} < 2 \snorm{A} \delta$. Meanwhile, by Gelfand's Formula,
  \begin{gather*}
    \snorm{M_1^t} = \snorm{P^{\top} A^t P} \leq \snorm{A^t} \leq \zeta_{\varepsilon_1}(A)(|\lambda_1| + \varepsilon_1)^t,\\
    \snorm{\hat{M}_1^t} = \snorm{\hat{P}^{\top} A^t \hat{P}} \leq \snorm{A^t} \leq \zeta_{\varepsilon_1}(A)(|\lambda_1| + \varepsilon_1)^t.
  \end{gather*}
  Then we have the following bound by telescoping
  \begin{align*}
    \snorm{M_1^{\tau} - \hat{M}_1^{\tau}}
    &= \norm{\sum_{i=1}^{\tau} \paren{M_1^{i} \hat{M}_1^{\tau-i} - M_1^{i-1} \hat{M}_1^{\tau-i+1}}} \\
    &\leq \sum_{i=1}^{\tau} \snorm{M_1^{i-1}} \snorm{\hat{M}_1^{\tau-i}} \snorm{M_1 - \hat{M}_1} \\
    &< \tau \cdot \zeta_{\varepsilon_1}(A)^2 (|\lambda_1| + \varepsilon_1)^{\tau-1} \cdot 2 \snorm{A} \delta \\
    &= 2 \tau \snorm{A} \zeta_{\varepsilon_1}(A)^2 (|\lambda_1| + \varepsilon_1)^{\tau-1} \delta.
  \end{align*}
  This finishes the proof.
\end{proof}

\begin{corollary}\label{thm:norm-bound-hat-M1-tau}
  Under the premises of Theorem \ref{thm:main-theorem}, when $\delta < \frac{1}{\tau}$, 
  $$\snorm{\hat{M}_1^{\tau}} < \big( \zeta_{\varepsilon_1}(M_1) (|\lambda_1| + \varepsilon_1) + 2 \snorm{A} \zeta_{\varepsilon_1}(A) \big) (|\lambda_1| + \varepsilon_1)^{\tau-1}.$$
\end{corollary}

\begin{proof}
  A combination of Gelfand's Formula and Proposition \ref{thm:estimation-error-M1-tau} yields
  \begin{align*}
    \snorm{\hat{M}_1^{\tau}}
    &\leq \snorm{M_1^{\tau}} + \snorm{\hat{M}_1^{\tau} - M_1^{\tau}} \\
    &\leq \zeta_{\varepsilon_1}(M_1) (|\lambda_1| + \varepsilon_1)^{\tau} + 2 \tau \snorm{A} \zeta_{\varepsilon_1}(A)^2 (|\lambda_1| + \varepsilon_1)^{\tau-1} \delta \\
    &< \big( \zeta_{\varepsilon_1}(M_1) (|\lambda_1| + \varepsilon_1) + 2 \tau \snorm{A} \zeta_{\varepsilon_1}(A) \delta \big) (|\lambda_1| + \varepsilon_1)^{\tau-1},
  \end{align*}
  where the last inequality requires $\delta < \frac{1}{\tau}$. This completes the proof.
\end{proof}

\subsection{Instability of the Unstable Component}\label{subsec:instability-unstable}

We have been referring to $E_{\mathrm{s}}$ (and approximately, $E_{\mathrm{u}}^{\perp}$) as ``stable'', and $E_{\mathrm{u}}$ as ``unstable''. This leads us to think that the unstable component will constitute an increasing proportion of the state as the system evolves with zero control input. However, in some cases it might happen that the proportion of unstable component does not increase within the first few time steps, although eventually it will explode. This motivates us to formally characterize such instability of the unstable component.

In this section, we aim to establish a fundamental property of $A^{\omega}$ (for large enough $\omega$, of course) that it ``almost surely'' increases the norm of the state. By ``almost surely'' we mean that the initial state should have non-negligible unstable component, which happens with probability $1-\varepsilon$ when we uniformly sample the initial state from the surface of unit hyper-sphere in $\R^n$.

Throughout this section, we use $\gamma$ to denote the ratio of the unstable component over the stable component within some state $x$ (i.e., $\frac{\snorm{R_1 x}}{\snorm{R_2 x}}$). Note that
$$x = \varPi_{\mathrm{u}} x + \varPi_{\mathrm{s}} x = Q_1 R_1 x + Q_2 R_2 x,$$
where $Q_1, Q_2$ are orthonormal. Hence
$$\snorm{R_1 x} - \snorm{R_2 x} \leq \snorm{x} \leq \snorm{R_1 x} + \snorm{R_2 x}.$$
As a consequence, when $\frac{\snorm{R_1 x}}{\snorm{R_2 x}} > \gamma > 1$, we also know that
$$\frac{\snorm{R_1 x}}{\snorm{x}} \geq \frac{\snorm{R_1 x}}{\snorm{R_1 x} + \snorm{R_2 x}} > \frac{\gamma}{\gamma+1},\quad
  \frac{\snorm{R_2 x}}{\snorm{x}} \leq \frac{\snorm{R_2 x}}{\snorm{R_1 x} - \snorm{R_2 x}} < \frac{1}{\gamma-1}.$$

The following results are presented to fit in the framework of an inductive proof. We first establish the inductive step, where Proposition \ref{thm:unstable-component-expansion} shows that the unstable component eventually becomes dominant with a non-negligible initial $\gamma$, and Proposition \ref{thm:unstable-component-control-input} shows that the unstable component will still constitute a non-negligible part after a control input of mild magnitude is injected. Meanwhile, Proposition \ref{thm:unstable-component-initial-sampling} shows that the initial unstable component is non-negligible with large probability.

\begin{proposition}\label{thm:unstable-component-expansion}
  Given a dynamical matrix $A$ and some constant $\gamma>0$, for any state $x$ such that $\frac{\snorm{R_1 x}}{\snorm{R_2 x}} > \gamma$, for any $\omega \in \N$, we have
  $$\frac{\snorm{R_1 A^{\omega} x}}{\snorm{R_2 A^{\omega} x}} >
    \gamma_{\omega} := C_{\gamma} \paren{\frac{|\lambda_k|}{(1 + \varepsilon_3 |\lambda_k|)(|\lambda_{k+1}|+\varepsilon_2)}}^{\omega},$$
  where $C_{\gamma} := \frac{1}{(1+\frac{1}{\gamma}) \zeta_{\varepsilon_3}(N_1^{-1}) \zeta_{\varepsilon_2}(N_2) \snorm{R_2}}$ is a constant related to $\gamma$. Specifically, for any $\gamma_{+} > 0$, there exists a constant $\omega_0(\gamma, \gamma_{+}) = O(\log \frac{\gamma_{+}}{\gamma})$, such that for any $\omega > \omega_0(\gamma, \gamma_{+})$, $\frac{\snorm{R_1 x}}{\snorm{R_2 x}} > \gamma_{+}$.
\end{proposition}

\begin{proof}
  Recall that $R_1 A^{\omega} = N_1^{\omega} R_1$ and $R_2 A^{\omega} = N_2^{\omega} R_2$. By Gelfand's Formula we have
  \begin{align*}
    \frac{\snorm{R_1 A^{\omega} x}}{\snorm{R_2 A^{\omega} x}}
    &= \frac{\snorm{N_1^{\omega} R_1 x}}{\snorm{N_2^{\omega} R_2 x}}
    \geq \frac{\sigma_{\min}(N_1^{\omega}) \snorm{R_1 x}}{\snorm{N_2^{\omega}} \snorm{R_2} \snorm{x}}
    > \frac{\sigma_{\min}(N_1^{\omega})}{(1+\frac{1}{\gamma}) \snorm{N_2^{\omega}} \snorm{R_2}} \\
    &\geq \frac{\big( |\lambda_k| / (1 + \varepsilon_3 |\lambda_k|) \big)^{\omega}}{(1+\frac{1}{\gamma}) \zeta_{\varepsilon_3}(N_1^{-1}) \zeta_{\varepsilon_2}(N_2) (|\lambda_{k+1}|+\varepsilon_2)^{\omega} \snorm{R_2}} \\
    &= \frac{1}{(1+\frac{1}{\gamma}) \zeta_{\varepsilon_3}(N_1^{-1}) \zeta_{\varepsilon_2}(N_2) \snorm{R_2}} \paren{\frac{|\lambda_k|}{(1 + \varepsilon_3 |\lambda_k|)(|\lambda_{k+1}|+\varepsilon_2)}}^{\omega}.
  \end{align*}
  Therefore, we shall take
  $$\omega_0(\gamma, \gamma_{+})
    = \frac{\log \gamma_{+} / C_{\gamma}}{\log (|\lambda_k|) / \big( (1 + \varepsilon_3 |\lambda_k|)(|\lambda_{k+1}|+\varepsilon_2) \big)}
    = O \paren{\log \frac{\gamma_{+}}{\gamma}},$$
  and the proof is completed.
\end{proof}

\begin{corollary}\label{thm:unstable-component-expansion-corollary}
  Under the premises of Proposition \ref{thm:unstable-component-expansion}, for any $\omega > \omega_0(\gamma, \gamma_{+})$, $$\frac{\snorm{P_1^{\top} A^{\omega} x}}{\snorm{A^{\omega} x}} > 1 - \frac{2}{\gamma_{\omega}-1},\quad
    \frac{\snorm{P_2^{\top} A^{\omega} x}}{\snorm{A^{\omega} x}} < \frac{1}{\gamma_{\omega}-1}.$$
\end{corollary}

\begin{proof}
  Note that we have decomposition $x = \varPi_{\mathrm{u}} x + \varPi_1 \varPi_{\mathrm{s}} x + \varPi_2 \varPi_{\mathrm{s}} x$, where $\snorm{\varPi_{\mathrm{u}} x} = \snorm{R_1 x}$ and $\snorm{\varPi_{\mathrm{s}} x} = \snorm{R_2 x}$. Hence, for any $\omega > \omega_0(\gamma, \gamma_{+})$, we can show that 
  \begin{align*}
    \frac{\snorm{P_1^{\top} A^{\omega} x}}{\snorm{A^{\omega} x}}
    &= \frac{\snorm{\varPi_{\mathrm{u}} A^{\omega} x + \varPi_1 \varPi_{\mathrm{s}} A^{\omega} x}}{\snorm{A^{\omega} x}} \\
    &\geq \frac{\snorm{\varPi_{\mathrm{u}} A^{\omega} x} - \snorm{\varPi_1 \varPi_{\mathrm{s}} A^{\omega} x}}{\snorm{A^{\omega} x}} \\
    &\geq \frac{\snorm{R_1 A^{\omega} x} - \snorm{R_2 A^{\omega} x}}{\snorm{A^{\omega} x}} \\
    &> \frac{\gamma_{\omega}}{\gamma_{\omega}+1} - \frac{1}{\gamma_{\omega}-1}
    > 1 - \frac{2}{\gamma_{\omega}-1},
  \end{align*}
  and similarly,
  $$\frac{\snorm{P_2^{\top} A^{\omega} x}}{\snorm{A^{\omega} x}}
    = \frac{\snorm{\varPi_2 \varPi_{\mathrm{s}} A^{\omega} x}}{\snorm{A^{\omega} x}}
    \leq \frac{\snorm{\varPi_{\mathrm{s}} A^{\omega} x}}{\snorm{A^{\omega} x}}
    < \frac{1}{\gamma_{\omega}-1}.$$
  The proof is completed.
\end{proof}

\begin{proposition}\label{thm:unstable-component-control-input}
  Given dynamical matrices $A, B$ and constants $\gamma>0, \gamma_{+}>1$, for any state $x$ such that $\frac{\snorm{R_1 x}}{\snorm{R_2 x}} > \gamma_{+}$, suppose we feed a control input $\snorm{u} \leq \alpha \snorm{x}$ and observe the next state $x' = Ax + Bu$, where $\alpha$ satisfies
  \begin{equation}\label{eq:constraint-alpha:1}
    \alpha < \frac{\frac{\gamma_{+}}{\gamma_{+}+1} \sigma_{\min}(M_1) - \frac{\gamma}{\gamma_{+}-1} \frac{1}{1-\xi}\snorm{A}}{(1+\frac{\sqrt{2\xi}}{1-\xi}+\frac{\gamma}{1-\xi}) \snorm{B}}.
  \end{equation}
  Then we can guarantee that $\frac{\snorm{R_1 x'}}{\snorm{R_2 x'}} > \gamma$.
\end{proposition}

\begin{proof}
  The proposition can be shown by direct calculation. Let $z = Rx = [z_1^{\top}, z_2^{\top}]^{\top}$. Recall that
  $$Rx' = z' = \begin{bmatrix}
    N_1 z_1 + R_1 Bu \\
    N_2 z_2 + R_2 Bu
  \end{bmatrix},$$
  and note that $\frac{\snorm{z_1}}{\snorm{x}} > \frac{\gamma_{+}}{\gamma_{+}+1}$, $\frac{\snorm{z_2}}{\snorm{x}} < \frac{1}{\gamma_{+}-1}$ under the assumptions, so we have
  \begin{align*}
    \frac{\snorm{R_1 x'}}{\snorm{R_2 x'}}
    &= \frac{\snorm{N_1 z_1 + R_1 Bu}}{\snorm{N_2 z_2 + R_2 Bu}}
    \geq \frac{\snorm{N_1 z_1} - \snorm{R_1 Bu}}{\snorm{N_2 z_2} + \snorm{R_2 Bu}} \\
    &\geq \frac{\sigma_{\min}(N_1) \snorm{z_1} - \snorm{R_1 B} \snorm{u}}{\snorm{N_2} \snorm{z_2} + \snorm{R_2 B} \snorm{u}} \\
    &\geq \frac{\sigma_{\min}(N_1) \frac{\gamma_{+}}{\gamma_{+}+1} \snorm{x} - \alpha \snorm{R_1} \snorm{B} \snorm{x}}{\snorm{N_2} \frac{1}{\gamma_{+}-1} \snorm{x} + \alpha \snorm{R_2} \snorm{B} \snorm{x}} \\
    &\geq \frac{\sigma_{\min}(M_1) \frac{\gamma_{+}}{\gamma_{+}+1} \snorm{x} - \alpha (1+\frac{\sqrt{2\xi}}{1-\xi}) \snorm{B} \snorm{x}}{\frac{1}{1-\xi}\snorm{A} \frac{1}{\gamma_{+}-1} \snorm{x} + \alpha \frac{1}{1-\xi} \snorm{B} \snorm{x}} \\
    &> \gamma,
  \end{align*}
  where we apply Lemma \ref{thm:lemma-xi-close-interpretation} and the convention of taking $N_1 = M_1.$
\end{proof}

\begin{proposition}\label{thm:unstable-component-initial-sampling}
  Suppose a state $x$ is sampled uniformly randomly from the unit hyper-sphere surface $\mathbb{B}_n \subset \R^n$, then for any constant $\gamma < \min \Big\lbrace \frac{1}{2}, \frac{1}{\sqrt{2 / (\sigma_{\min}(R_1) k)} + 1} \Big\rbrace$, we have
  $$\Prob[x \sim \mathcal{U}(\mathbb{B}_n)]{\frac{\snorm{R_1 x}}{\snorm{R_2 x}} > \gamma} > 1 - \theta(\gamma),$$
  where $\theta(\gamma) = \frac{8\sqrt{2}}{\mathrm{B}(\frac{1}{2}, \frac{n-1}{2}) \sqrt{\sigma_{\min}(R_1)}} \gamma = O(\gamma)$ is a constant bounded linearly by $\gamma$.
\end{proposition}

\begin{proof}
  Note that
  $$\snorm{R_1 x} > \frac{\gamma}{1-\gamma} \snorm{x}
    ~\Rightarrow~ \snorm{R_2 x} < \snorm{x} + \snorm{R_1 x} < \frac{1}{1-\gamma} \snorm{x}
    ~\Rightarrow~ \frac{\snorm{R_1 x}}{\snorm{R_2 x}} > \gamma.$$
  so we only have to show that $\Prob[x \sim \mathcal{U}(\mathbb{B}_n)]{\snorm{R_1 x} \leq \frac{\gamma}{1-\gamma}} < \theta(\gamma)$. Now let $R_1^{\top} R_1 = S^{\top} D S$ be the eigen-decomposition of $R_1^{\top} R_1$, where $S$ is selected to be orthonormal such that
  $$D = \diag(d_1, \cdots, d_k, 0, \cdots, 0).$$
  Note that the vector $y = Sx =: [y_1, \cdots, y_n]$ also obeys a uniform distribution over $\mathbb{B}_n$, so we have
  \begin{align*}
      \Prob{\snorm{R_1 x} \leq \tfrac{\gamma}{1-\gamma}}
      &= \Prob{x^\top R_1^\top R_1 x \leq (\tfrac{\gamma}{1-\gamma})^2}
      = \Prob{y^\top D y \leq (\tfrac{\gamma}{1-\gamma})^2} \\
      &\leq \Prob{d_i y_i^2 \leq \tfrac{1}{k} (\tfrac{\gamma}{1-\gamma})^2,~ \forall i=1,\ldots,k} \\
      &\leq \sum_{i=1}^{k} \Prob{y_i^2 \leq \tfrac{1}{d_i k}(\tfrac{\gamma}{1-\gamma})^2}.
  \end{align*}
  It suffices to bound the probability $\Prob[y \sim \mathcal{U}(B)]{y_i^2 \leq \eta}$. Note that $y$ can be obtained by first sampling a Gaussian random vector $z \sim \mathcal{N}(0, I_n)$, and then normalize it to get $y = \frac{z}{\snorm{z}}$. Hence
  $$\Prob[y \sim \mathcal{U}(\mathbb{B}_n)]{y_i^2 \leq \eta}
    = \Prob[z \sim \mathcal{N}(0, I_n)]{z_i^2 \leq \eta \snorm{z}^2}
    = \Prob[z \sim \mathcal{N}(0, I_n)]{\frac{z_i^2}{\sum_{j \neq i} z_j^2} \leq \frac{\eta}{1-\eta}},$$
  where $w := \frac{z_i^2}{\sum_{j \neq i} z_j^2}$ is known to obey an F-distribution $w \sim \mathcal{F}(1, n-1)$. The c.d.f. of $w$ is known to be $I_{w/(w+n-1)}(\frac{1}{2}, \frac{n-1}{2})$, where $I$ denotes the \textit{regularized incomplete Beta function}. Note that
  $$I_{w/(w+n-1)} \paren{\frac{1}{2}, \frac{n-1}{2}} = \frac{2 w^{1/2}}{(n-1)^{1/2} \mathrm{B}(\frac{1}{2}, \frac{n-1}{2})} - \frac{n w^{3/2}}{3 (n-1)^{3/2} \mathrm{B}(\frac{1}{2}, \frac{n-1}{2})} + O(n^{5/2}),$$
  it can be shown that $I_{w/(w+n-1)} \paren{\frac{1}{2}, \frac{n-1}{2}} < \frac{4 \sqrt{w}}{\sqrt{n-1} \mathrm{B}(\frac{1}{2}, \frac{n-1}{2})}$. Hence
  $$\Prob[y \sim \mathcal{U}(\mathbb{B}_n)]{y_i^2 \leq \eta}
    = \Prob[z \sim \mathcal{N}(0, I_n)]{\frac{z_i^2}{\sum_{j \neq i} z_j^2} \leq \frac{\eta}{1-\eta}}
    < \frac{4 \sqrt{\frac{\eta}{1-\eta}}}{\sqrt{n-1} \mathrm{B}(\frac{1}{2}, \frac{n-1}{2})},$$
  which further gives
  $$\Prob{\snorm{R_1 x} \leq \tfrac{\gamma}{1-\gamma}}
    < \sum_{i=1}^{k} \frac{4 \sqrt{\frac{2}{d_i k}(\tfrac{\gamma}{1-\gamma})^2}}{\sqrt{n-1} \mathrm{B}(\frac{1}{2}, \frac{n-1}{2})}
    < \frac{8\sqrt{2}}{\mathrm{B}(\frac{1}{2}, \frac{n-1}{2}) \sqrt{\sigma_{\min}(R_1)}} \gamma
    = O(\gamma)$$
  where we require $\gamma < \min \Big\lbrace \frac{1}{2}, \frac{1}{\sqrt{2 / (\sigma_{\min}(R_1) k)} + 1} \Big\rbrace$.
\end{proof}

Combining the previous three propositions, we have shown in an inductive way that the algorithm guarantees $\frac{\snorm{P_2^{\top} x_{t_i}}}{\snorm{x_{t_i}}}$ is constantly upper bounded at each time step $t_i$ ($i=1,\cdots,k$), which is critical to the estimation error bound of $B_{\tau}$. This is concluded as the following lemma.

\begin{lemma}\label{thm:lemma-unstable-expansion}
  Under the premises of Theorem \ref{thm:main-theorem}, for any constant $\gamma < \min \Big\lbrace \frac{1}{2}, \frac{1}{\sqrt{2 / (\sigma_{\min}(R_1) k)} + 1} \Big\rbrace$ and $\gamma < t_0$, the algorithm guarantees
  $$\frac{\snorm{P_2^{\top} x_{t_i}}}{\snorm{x_{t_i}}} < \frac{1}{\gamma_{\omega}-1},~\forall i=1,\cdots,k$$
  with probability $1-\theta(\gamma)$ over the initialization of $x_0$ on the unit hyper-sphere surface $\mathbb{B}_n$, where
  $$\gamma_{\omega} := C_{\gamma} \paren{\frac{|\lambda_k|}{(1 + \varepsilon_3 |\lambda_k|)(|\lambda_{k+1}|+\varepsilon_2)}}^{\omega}.$$
\end{lemma}

\begin{proof}
  We proceed by showing that $\frac{\snorm{R_1 x_{t_i}}}{\snorm{R_2 x_{t_i}}} > \gamma_{\omega}$ for $i=1,\cdots,k$ in an inductive way.
  
  For the base case, it is guaranteed by Proposition \ref{thm:unstable-component-initial-sampling} that $x_0$ satisfies $\frac{\snorm{R_1 x_0}}{\snorm{R_2 x_0}} > \gamma$ with probability $1 - \theta(\gamma)$, and Proposition \ref{thm:unstable-component-expansion} further guarantees $\frac{\snorm{R_1 x_{t_1}}}{\snorm{R_2 x_{t_1}}} > \gamma_{\omega}$. Here we require $t_0 > \omega$.
  
  For the inductive step, suppose we have shown $\frac{\snorm{R_1 x_{t_i}}}{\snorm{R_2 x_{t_i}}} > \gamma_{\omega}$. Since $\snorm{u_{t_i}} = \alpha \snorm{x_{t_i}}$, we have $\frac{\snorm{R_1 x_{t_i+1}}}{\snorm{R_2 x_{t_i+1}}} > \gamma$ by Proposition \ref{thm:unstable-component-control-input}, and again Proposition \ref{thm:unstable-component-expansion} guarantees $\frac{\snorm{R_1 x_{t_{i+1}}}}{\snorm{R_2 x_{t_{i+1}}}} > \gamma_{\omega}$.
  
  Now it only suffices to apply Corollary \ref{thm:unstable-component-expansion-corollary} to complete the proof.
\end{proof}

\subsection{Estimation Error of \texorpdfstring{\titlemath{B_{\tau}}}{B\_τ}}

\begin{proposition}\label{thm:estimation-error-B-general-A}
  Under the premises of Theorem \ref{thm:main-theorem} and Lemma \ref{thm:lemma-unstable-expansion}, when (\ref{eq:constraint-omega:1}) holds,
  $$\snorm{\hat{B}_{\tau} - B_{\tau}} < C_{B} (|\lambda_1| + \varepsilon_1)^{\tau-1} \delta,$$
  where $C_{B} := \frac{2\sqrt{k} \zeta_{\varepsilon_1}(A)^2 \big( (2\tau+2)\snorm{A} + \snorm{B} \big)}{\alpha}$.
\end{proposition}

\begin{proof}
  This is parallel to Lemma \ref{thm:estimation-error-B-symmetric-A}. Note that we have to subtract an additional term (induced by non-zero $\varDelta_{\tau}$ in $M^{\tau}$) to calculate the actual $b_i$, so we have
  \begin{align*}
    \snorm{b_i - \hat{b}_i}
    &= \frac{1}{\alpha \snorm{x_{t_i}}} \norm{\big( P_1^{\top} x_{t_i+\tau} - M_1^{\tau} P_1^{\top} x_{t_i} - \varDelta_{\tau} P_2^{\top} x_{t_i} \big) - \big( \hat{P}_1^{\top} x_{t_i+\tau} - \hat{M}_1^{\tau} \hat{P}_1^{\top} x_{t_i} \big)} \\
    &\leq \frac{1}{\alpha \snorm{x_{t_i}}} \paren{\snorm{(P_1 - \hat{P}_1)^{\top} (A^{\tau} x_{t_i} + B_{\tau} u_{t_i})} + \snorm{M_1^{\tau} P_1^{\top} x_{t_i} - \hat{M}_1^{\tau} \hat{P}_1^{\top} x_{t_i}} + \snorm{\varDelta_{\tau} P_2^{\top} x_{t_i}}} \\
    &< \frac{1}{\alpha} \paren{\zeta_{\varepsilon_1}(A)^2 (|\lambda_1| + \varepsilon_1)^{\tau-1} \big( (2\tau+2)\snorm{A} + \snorm{B} \big) \delta + \delta}.
  \end{align*}
  Here the first term is bounded by
  \begin{align*}
    \snorm{(P_1 - \hat{P}_1)^{\top} (A^{\tau} x_{t_i} + B_{\tau} u_{t_i})}
    &\leq \snorm{P_1 - \hat{P}_1} (\snorm{A^{\tau}} + \snorm{A^{\tau-1} B}) \snorm{x_{t_i}} \\
    &< \snorm{x_{t_i}} \zeta_{\varepsilon_1}(A) (|\lambda_1| + \varepsilon_1)^{\tau-1} (\snorm{A} + \snorm{B}) \delta,
  \end{align*}
  where in the last inequality we apply Corollary \ref{thm:projector-error-corollary}; the second term is bounded by
  \begin{align}
    \snorm{M_1^{\tau} P_1^{\top} x_{t_i} - \hat{M}_1^{\tau} \hat{P}_1^{\top} x_{t_i}}
    &\leq (\snorm{M_1^{\tau}(P_1^{\top} - \hat{P}_1^{\top})} + \snorm{(M_1^{\tau} - \hat{M}_1^{\tau}) \hat{P}_1^{\top}}) \snorm{x_{t_i}} \nonumber\\
    &< \big( \zeta_{\varepsilon_1}(A) (|\lambda_1| + \varepsilon_1)^{\tau-1} \snorm{A} \delta \nonumber\\
      &\qquad + 2 \tau \snorm{A} \zeta_{\varepsilon_1}(A)^2 (|\lambda_1| + \varepsilon_1)^{\tau-1} \delta \big) \snorm{x_{t_i}} \label{eq:estimation-error-B-1:1}\\
    &\leq \snorm{x_{t_i}} \zeta_{\varepsilon_1}(A)^2 (|\lambda_1| + \varepsilon_1)^{\tau-1} (2\tau+1) \snorm{A} \delta, \label{eq:estimation-error-B-1:2}
  \end{align}
  where in (\ref{eq:estimation-error-B-1:1}) we apply Proposition \ref{thm:estimation-error-M1-tau}, and in (\ref{eq:estimation-error-B-1:2}) we apply a simple fact that $\zeta_{\varepsilon_1}(A) \geq 1$; the third term is bounded by 
  \begin{align}
    \frac{\snorm{\varDelta_{\tau}} \snorm{P_2^{\top} x_{t_i}}}{ \snorm{x_{t_i}}}
    &\leq \frac{C_{\varDelta}(|\lambda_1| + \varepsilon_1)^{\tau}}{ \bracket{C_{\gamma} \paren{\frac{|\lambda_k|}{(1 + \varepsilon_3 |\lambda_k|)(|\lambda_{k+1}|+\varepsilon_2)}}^{\omega} - 1}} \label{eq:estimation-error-B-2:1}\\
    &< \frac{2 C_{\varDelta} (|\lambda_1| + \varepsilon_1)^{\tau}}{C_{\gamma} \paren{\frac{|\lambda_k|}{(1 + \varepsilon_3 |\lambda_k|)(|\lambda_{k+1}|+\varepsilon_2)}}^{\omega}} \label{eq:estimation-error-B-2:2}\\
    &< \delta, \label{eq:estimation-error-B-2:3}
  \end{align}
  where in (\ref{eq:estimation-error-B-2:1}) we apply Lemma \ref{thm:lemma-unstable-expansion}, while in (\ref{eq:estimation-error-B-2:2}) and (\ref{eq:estimation-error-B-2:3}) we require
  \begin{equation}\label{eq:constraint-omega:1}
    \omega > \max\braced{\frac{\log 2/C_{\gamma}}{\log \big( |\lambda_k| / (1 + \varepsilon_3 |\lambda_k|)(|\lambda_{k+1}|+\varepsilon_2) \big)}, \frac{\log (2 C_{\varDelta}) / (C_{\gamma} \delta) + \tau \log (|\lambda_1| + \varepsilon_1)}{\log \big( |\lambda_k| / (1 + \varepsilon_3 |\lambda_k|)(|\lambda_{k+1}|+\varepsilon_2) \big)}}.
  \end{equation}
  
  Finally, to bound the error of the whole matrix, we simply apply the definition
  \begin{align*}
    \snorm{\hat{B}_{\tau} - B_{\tau}}
    &= \max_{\snorm{u}=1} \snorm{(\hat{B}_{\tau} - B_{\tau})u}
    \leq \max_{\snorm{u}=1} \sum_{i=1}^{k} |u_i| \snorm{\hat{b}_i - b_i} \\
    &< \frac{\sqrt{k}}{\alpha} \paren{\zeta_{\varepsilon_1}(A)^2 (|\lambda_1| + \varepsilon_1)^{\tau-1} \big( (2\tau+2)\snorm{A} + \snorm{B} \big) + 1} \delta \\
    &< \frac{2\sqrt{k} \zeta_{\varepsilon_1}(A)^2 \big( (2\tau+2)\snorm{A} + \snorm{B} \big)}{\alpha} (|\lambda_1| + \varepsilon_1)^{\tau-1} \delta.
  \end{align*}
  This completes the proof.
\end{proof}

\begin{corollary}\label{thm:norm-bound-hat-B-tau}
  Under the premises of Theorem \ref{thm:main-theorem} and Lemma \ref{thm:lemma-unstable-expansion}, when (\ref{eq:constraint-omega:1}), (\ref{eq:constraint-tau:1}) and (\ref{eq:constraint-delta:1}) hold,
  $$\sigma_{\min}(\hat{B}_{\tau}) > \frac{c \snorm{B}}{4 \zeta_{\varepsilon_3}(N_1^{-1})} \paren{\frac{|\lambda_k|}{1 + \varepsilon_3 |\lambda_k|}}^{\tau-1}.$$
\end{corollary}

\begin{proof}
  We apply the $E_{\mathrm{u}} \oplus E_{\mathrm{s}}$-decomposition. Note that
  $$B_{\tau}
    = P_1^{\top} A^{\tau-1} B
    = P_1^{\top} (Q_1 N_1^{\tau-1} R_1 + Q_2 N_2^{\tau-1} R_2) B
    = N_1^{\tau-1} R_1 B + P_1^{\top} Q_2 N_2^{\tau-1} R_2 B,$$
  so by Gelfand's Formula and Lemma \ref{thm:lemma-xi-close-interpretation} we have
  \begin{align*}
    \sigma_{\min}(B_{\tau})
    &= \sigma_{\min} (N_1^{\tau-1} R_1 B + P_1^{\top} Q_2 N_2^{\tau-1} R_2 B) \\
    &\geq \sigma_{\min}(N_1^{\tau-1}) \sigma_{\min}(R_1 B) - \snorm{P_1^{\top} Q_2} \snorm{N_2^{\tau-1}} \snorm{R_2} \snorm{B} \\
    &\geq \frac{c \snorm{B}}{\zeta_{\varepsilon_3}(N_1^{-1})} \paren{\frac{|\lambda_k|}{1 + \varepsilon_3 |\lambda_k|}}^{\tau-1} - \frac{\sqrt{2\xi} \zeta_{\varepsilon_2}(N_2) \snorm{B}}{1-\xi} (|\lambda_{k+1}| + \varepsilon_2)^{\tau-1} \\
    &> \frac{c \snorm{B}}{2 \zeta_{\varepsilon_3}(N_1^{-1})} \paren{\frac{|\lambda_k|}{1 + \varepsilon_3 |\lambda_k|}}^{\tau-1}
  \end{align*}
  where the last inequality requires
  $$\frac{\sqrt{2\xi} \zeta_{\varepsilon_2}(N_2) \zeta_{\varepsilon_3}(N_1^{-1})}{c(1-\xi)} \paren{\frac{ (|\lambda_{k+1}| + \varepsilon_2) (1+\varepsilon_3 |\lambda_k|)}{|\lambda_k|}}^{\tau-1} < \frac{1}{2},$$
  or equivalently,
  \begin{equation}\label{eq:constraint-tau:1}
    \tau > \frac{\log \frac{c(1-\xi)}{2\sqrt{2\xi} \zeta_{\varepsilon_2}(N_2) \zeta_{\varepsilon_3}(N_1^{-1})}}{\log \frac{ (|\lambda_{k+1}| + \varepsilon_2) (1+\varepsilon_3 |\lambda_k|)}{|\lambda_k|}} + 1.
  \end{equation}
  Therefore, using Proposition \ref{thm:estimation-error-B-general-A}, $\sigma_{\min}(\hat{B}_{\tau})$ is lower bounded by
  \begin{align*}
    \sigma_{\min}(\hat{B}_{\tau})
    &\geq \sigma_{\min}(B_{\tau}) - \snorm{\hat{B}_{\tau} - B_{\tau}} \\
    &> \frac{c \snorm{B}}{2 \zeta_{\varepsilon_3}(N_1^{-1})} \paren{\frac{|\lambda_k|}{1 + \varepsilon_3 |\lambda_k|}}^{\tau-1} - C_{B} (|\lambda_1| + \varepsilon_1)^{\tau-1} \delta \\
    &> \frac{c \snorm{B}}{4 \zeta_{\varepsilon_3}(N_1^{-1})} \paren{\frac{|\lambda_k|}{1 + \varepsilon_3 |\lambda_k|}}^{\tau-1},
  \end{align*}
  where the last inequality requires
  \begin{equation}\label{eq:constraint-delta:1}
    \delta < \frac{c \snorm{B}}{4 \zeta_{\varepsilon_3}(N_1^{-1}) C_{B}} \paren{\frac{|\lambda_k|}{(1 + \varepsilon_3 |\lambda_k|) (|\lambda_1| + \varepsilon_1)}}^{\tau-1}.
  \end{equation}
  This completes the proof.
\end{proof}

Finally, using the above bounds, we can easily upper bound the norm of our controller $\hat{K}_1$.

\begin{proposition}\label{thm:upper-bound-K1}
  Under the premises of Theorem \ref{thm:main-theorem}, when (\ref{eq:constraint-omega:1}), (\ref{eq:constraint-tau:1}), (\ref{eq:constraint-delta:1}) and $\delta < \frac{1}{\tau}$ hold,
  $$\snorm{\hat{K}_1}
    < C_{K} \paren{\frac{(|\lambda_1| + \varepsilon_1) (1+\varepsilon_3 |\lambda_k|)}{|\lambda_k|}}^{\tau-1},$$
  where $C_{K} := \frac{4 \zeta_{\varepsilon_3}(N_1^{-1}) \big( \zeta_{\varepsilon_1}(M_1) (|\lambda_1| + \varepsilon_1) + 2 \snorm{A} \zeta_{\varepsilon_1}(A) \big)}{c \snorm{B}}$.
\end{proposition}

\begin{proof}
  Recall that the controller is constructed as $\hat{K}_1 = \hat{B}_{\tau}^{-1} \hat{M}_1^{\tau} \hat{P}_1^{\top}$, so we have
  $$\snorm{\hat{K}_1}
    \leq \snorm{\hat{B}_{\tau}^{-1}} \snorm{\hat{M}_1^{\tau}}
    = \frac{\snorm{\hat{M}_1^{\tau}}}{\sigma_{\min}(\hat{B}_{\tau})},$$
  and the bound is merely a combination of Corollary \ref{thm:norm-bound-hat-M1-tau} and Corollary \ref{thm:norm-bound-hat-B-tau} whenever $\delta < \frac{1}{\tau}$.
\end{proof}

\subsection{Proof of Theorem \ref{thm:main-theorem}}

Now we are ready to combine the above building blocks and present the complete proof of Theorem \ref{thm:main-theorem}. Note that, with all the bounds established above, the proof structure parallels that of Theorem \ref{thm:main-theorem-symmetric-A}, the special case with a symmetric dynamical matrix $A$.

\begin{proof}[Proof of Theorem \ref{thm:main-theorem}]
  The proof is again based on Lemma \ref{thm:block-estimate-radius}. We first guarantee that the diagonal blocks are stable. For the top-left block,
  \begin{align}
    \snorm{M_1^{\tau} + P_1^{\top} A^{\tau-1}B\hat{K}_1}
    &= \snorm{M_1^{\tau} - B_{\tau} \hat{B}_{\tau}^{-1} \hat{M}_1^{\tau} \hat{P}_1^{\top} P_1} \nonumber\\
    &\leq \snorm{M_1^{\tau} - \hat{M}_1^{\tau}} + \snorm{(B_{\tau} - \hat{B}_{\tau}) \hat{B}_{\tau}^{-1} \hat{M}_1^{\tau}} + \snorm{B_{\tau} \hat{B}_{\tau}^{-1} \hat{M}_1^{\tau} (I - \hat{P}_1^{\top} P_1)} \nonumber\\
    &\leq \snorm{M_1^{\tau} - \hat{M}_1^{\tau}} + \snorm{B_{\tau} - \hat{B}_{\tau}} \snorm{\hat{K}_1} + \snorm{B_{\tau}} \snorm{\hat{K}_1} \snorm{I - \hat{P}_1^{\top} P_1} \nonumber\\
    &\leq 2 \tau \snorm{A} \zeta_{\varepsilon_1}(A)^2 (|\lambda_1| + \varepsilon_1)^{\tau-1} \delta \nonumber\\
      &\qquad + C_{B} C_{K} \paren{\frac{(|\lambda_1| + \varepsilon_1)^2 (1+\varepsilon_3 |\lambda_k|)}{|\lambda_k|}}^{\tau-1} \delta \label{eq:proof-main-block-11:1}\\
      &\qquad + \zeta_{\varepsilon_1}(A) \snorm{B} C_{K} \paren{\frac{(|\lambda_1| + \varepsilon_1)^2 (1+\varepsilon_3 |\lambda_k|)}{|\lambda_k|}}^{\tau-1} \delta \nonumber\\
    &< ( C_{B} C_{K} + \zeta_{\varepsilon_1}(A) \snorm{B} C_{K} + 1 ) \paren{\frac{(|\lambda_1| + \varepsilon_1)^2 (1+\varepsilon_3 |\lambda_k|)}{|\lambda_k|}}^{\tau-1} \delta \label{eq:proof-main-block-11:2}\\
    &< \frac{1}{2}, \label{eq:proof-main-block-11:3}
  \end{align}
  where in (\ref{eq:proof-main-block-11:1}) we apply Propositions \ref{thm:estimation-error-M1-tau}, \ref{thm:estimation-error-B-general-A}, \ref{thm:upper-bound-K1}, and \ref{prop:norm-bound-projection-error}; in (\ref{eq:proof-main-block-11:2}) we require
  \begin{equation}\label{eq:constraint-tau:2}
    \frac{1}{\tau} \paren{\frac{(|\lambda_1| + \varepsilon_1)^2 (1+\varepsilon_3 |\lambda_k|)}{|\lambda_k|}}^{\tau-1} > 2\snorm{A} \zeta_{\varepsilon_1}(A)^2;
  \end{equation}
  and in (\ref{eq:proof-main-block-11:3}) we require
  \begin{equation}\label{eq:constraint-delta:2}
    \delta < \frac{1}{2(C_{B} C_{K} + \zeta_{\varepsilon_1}(A) \snorm{B} C_{K} + 1)} \paren{\frac{(|\lambda_1| + \varepsilon_1)^2 (1+\varepsilon_3 |\lambda_k|)}{|\lambda_k|}}^{-(\tau-1)}.
  \end{equation}
  For the bottom-right block, it is straight-forward to see that
  \begin{align*}
    \snorm{M_2^{\tau} + P_2^{\top} A^{\tau-1}B \hat{K}_1 \hat{P}_1^{\top} P_2}
    &\leq \snorm{M_2^{\tau}} + \snorm{P_2^{\top} A^{\tau-1}} \snorm{B} \snorm{\hat{K}_1} \snorm{\hat{P}_1^{\top} P_2} \\
    &\leq \zeta_{\varepsilon_2}(M_2) (|\lambda_{k+1}| + \varepsilon_2)^{\tau} \\
      &\qquad + \zeta_{\varepsilon_2}(M_2) \snorm{B} C_{K} \paren{\frac{(|\lambda_1| + \varepsilon_1) (|\lambda_{k+1}| + \varepsilon_2) (1+\varepsilon_3 |\lambda_k|)}{|\lambda_k|}}^{\tau-1} \delta \\
    &< 1
  \end{align*}
  where the last inequality requires
  \begin{gather}
    \tau > \frac{\log 1 / (4\zeta_{\varepsilon_2}(M_2))}{\log (|\lambda_{k+1}| + \varepsilon_2)}, \label{eq:constraint-tau:3}\\
    \delta < \frac{1}{4 \zeta_{\varepsilon_2}(M_2) \snorm{B} C_{K}} \paren{\frac{(|\lambda_1| + \varepsilon_1) (|\lambda_{k+1}| + \varepsilon_2) (1+\varepsilon_3 |\lambda_k|)}{|\lambda_k|}}^{-(\tau-1)}. \label{eq:constraint-delta:3}
  \end{gather}
  
  Now it only suffices to bound the spectral norms of off-diagonal blocks. Note that, by applying Proposition \ref{thm:upper-bound-K1} and Proposition \ref{prop:norm-bound-Gelfand}, the top-right block is bounded as
  \begin{align*}
    \snorm{\varDelta_{\tau} + P_1^{\top} A^{\tau-1}B \hat{K}_1 \hat{P}_1^{\top} P_2}
    &\leq \snorm{\varDelta_{\tau}} + \snorm{B_{\tau}} \snorm{\hat{K}_1} \snorm{\hat{P}_1^{\top} P_2} \\
    &< C_{\varDelta} (|\lambda_1| + \varepsilon_1)^{\tau} \\
      &\qquad + \zeta_{\varepsilon_1}(A) \snorm{B} C_{K} \paren{\frac{(|\lambda_1| + \varepsilon_1)^2 (1+\varepsilon_3 |\lambda_k|)}{|\lambda_k|}}^{\tau-1} \delta \\
    &< (C_{\varDelta} + 1) (|\lambda_1| + \varepsilon_1)^{\tau}
  \end{align*}
  where the last inequality requires
  \begin{equation}\label{eq:constraint-delta:4}
    \delta < \frac{(|\lambda_1| + \varepsilon_1)^2}{\zeta_{\varepsilon_1}(A) \snorm{B} C_{K}} \paren{\frac{(|\lambda_1| + \varepsilon_1)^2 (1+\varepsilon_3 |\lambda_k|)}{|\lambda_k|}}^{-\tau};
  \end{equation}
  and the bottom-left block is bounded as
  \begin{align*}
    \snorm{P_2^{\top} A^{\tau-1}B\hat{K}_1 \hat{P}_1^{\top} P_1}
    &\leq \snorm{P_2^{\top} A^{\tau-1}} \snorm{B} \snorm{\hat{K}_1} \\
    &< \zeta_{\varepsilon_2}(M_2) \snorm{B} C_{K} \paren{\frac{(|\lambda_1| + \varepsilon_1) (|\lambda_{k+1}| + \varepsilon_2) (1+\varepsilon_3 |\lambda_k|)}{|\lambda_k|}}^{\tau-1}.
  \end{align*}
  
  Now, by Lemma \ref{thm:block-estimate-radius}, we can guarantee that
  $$\rho(\hat{L}_{\tau})
    \leq \frac{1}{2} + \chi(\hat{L}_{\tau}) \frac{(C_{\varDelta} + 1) \zeta_{\varepsilon_2}(M_2) \snorm{B} C_{K}}{|\lambda_1| + \varepsilon_1} \paren{\frac{(|\lambda_1| + \varepsilon_1)^2 (|\lambda_{k+1}| + \varepsilon_2) (1+\varepsilon_3 |\lambda_k|)}{|\lambda_k|}}^{\tau-1}
    < 1,$$
  which requires
  \begin{equation}\label{eq:constraint-tau:4}
    \tau > \frac{\log \frac{2(|\lambda_1| + \varepsilon_1)}{\chi(\hat{L}_{\tau}) (C_{\varDelta} + 1) \zeta_{\varepsilon_2}(M_2) \snorm{B} C_{K}}}{\log \frac{(|\lambda_1| + \varepsilon_1)^2 (|\lambda_{k+1}| + \varepsilon_2) (1+\varepsilon_3 |\lambda_k|)}{|\lambda_k|}}.
  \end{equation}
  Note that the above constraint makes sense only if $|\lambda_1|^2 |\lambda_{k+1}| < |\lambda_k|$.
  
  So far, it is still left to recollect all the constraints we need on the parameters $\tau, \alpha, \delta, \gamma$ and $\omega$. To start with, all constraints on $\tau$ (see (\ref{eq:constraint-tau:1}), (\ref{eq:constraint-tau:2}), (\ref{eq:constraint-tau:3}) and (\ref{eq:constraint-tau:4})) can be summarized as
  \begin{align*}
    \tau &> \max \left\lbrace
      \frac{\log \frac{c(1-\xi)}{2\sqrt{2\xi} \zeta_{\varepsilon_2}(N_2) \zeta_{\varepsilon_3}(N_1^{-1})}}{\log \frac{ (|\lambda_{k+1}| + \varepsilon_2) (1+\varepsilon_3 |\lambda_k|)}{|\lambda_k|}} + 1,
      \frac{\log 1 / (4\zeta_{\varepsilon_2}(M_2))}{\log (|\lambda_{k+1}| + \varepsilon_2)},
      \frac{\log \frac{2(|\lambda_1| + \varepsilon_1)}{\chi(\hat{L}_{\tau}) (C_{\varDelta} + 1) \zeta_{\varepsilon_2}(M_2) \snorm{B} C_{K}}}{\log \frac{(|\lambda_1| + \varepsilon_1)^2 (|\lambda_{k+1}| + \varepsilon_2) (1+\varepsilon_3 |\lambda_k|)}{|\lambda_k|}}, \right.\\
      &\qquad\qquad\qquad\left. - \frac{1}{\log \frac{(|\lambda_1| + \varepsilon_1)^2 (1+\varepsilon_3 |\lambda_k|)}{|\lambda_k|}} W_{-1} \paren{- \frac{\log \frac{(|\lambda_1| + \varepsilon_1)^2 (1+\varepsilon_3 |\lambda_k|)}{|\lambda_k|}}{2\snorm{A} \zeta_{\varepsilon_1}(A)^2 \frac{(|\lambda_1| + \varepsilon_1)^2 (1+\varepsilon_3 |\lambda_k|)}{|\lambda_k|}}}
    \right\rbrace,
  \end{align*}
  where $W_{-1}$ denotes the non-principle branch of the Lambert-W function. Here we utilize the fact that, for $x > \frac{1}{\log a}$, $y = \frac{a^x}{x}$ is monotone increasing with inverse function $x = - \frac{1}{\log a} W_{-1} \big( - \frac{\log a}{y} \big)$, which can be upper bounded by Theorem 1 in \cite{Ioannis2013Lambert} as
  $$- \frac{1}{\log a} W_{-1} \paren{- \frac{\log a}{y}}
    < \frac{\log y - \log \log a + \sqrt{2(\log y - \log \log a)}}{\log a}
    < \frac{3(\log y - \log \log a)}{\log a}.$$
  By gathering different constants, we have
  \begin{equation}\label{eq:constraint-tau:merge}
    \tau > \frac{\log \frac{\sqrt{\xi}}{1-\xi} + \log \frac{1}{c} + \log \chi(\hat{L}_{\tau}) + 5 \log \bar{\zeta} + \log \frac{\snorm{A}}{|\lambda_1| - |\lambda_{k+1}|} + C_{\tau}}{\log \frac{|\lambda_k|}{|\lambda_1|^2 |\lambda_{k+1}|}}
    = O(1),
  \end{equation}
  where we define $\bar{\zeta} := \max\sbraced{\zeta_{\varepsilon_1}(A), \zeta_{\varepsilon_2}(M_2), \zeta_{\varepsilon_2}(N_2), \zeta_{\varepsilon_3}(N_1^{-1})}$, and $C_{\tau}$ is a numerical constant.
  Note that we have to guarantee the denominator to be positive, which gives rise to the additional assumption $|\lambda_1|^2 |\lambda_{k+1}| < |\lambda_k|$.
  Meanwhile, for any $\ell \in \N$, we shall select $\gamma$ such that
  \begin{equation}\label{eq:constraint-gamma:merge}
    \gamma = O(k^{-\ell}),\quad
    \gamma < \min \braced{\frac{1}{2}, \frac{1}{\sqrt{2 / (\sigma_{\min}(R_1) k)} + 1}},
  \end{equation}
  and select $\alpha$ such that (see (\ref{eq:constraint-alpha:1}), and we have already guaranteed $\gamma_{\omega} > 2$ in (\ref{eq:constraint-omega:1}))
  \begin{equation}\label{eq:constraint-alpha:merge}
    \alpha < \frac{\frac{2}{3} \sigma_{\min}(M_1) - \frac{\gamma}{1-\xi}\snorm{A}}{(1 + \frac{\sqrt{2\xi}}{1-\xi} + \frac{\gamma}{1-\xi}) \snorm{B}} = O(1).
  \end{equation}
  Now constraints on $\delta$ (see (\ref{eq:constraint-delta:1}), (\ref{eq:constraint-delta:2}), (\ref{eq:constraint-delta:3}) and (\ref{eq:constraint-delta:4})) can be summarized as
  \begin{align*}
    \delta &< \min \left\lbrace
      \frac{c \snorm{B}}{4 \zeta_{\varepsilon_3}(N_1^{-1}) C_{B}} \paren{\frac{|\lambda_k|}{(1 + \varepsilon_3 |\lambda_k|) (|\lambda_1| + \varepsilon_1)}}^{\tau-1}, \right. \\
      &\qquad\qquad\qquad \frac{1}{2(C_{B} C_{K} + \zeta_{\varepsilon_1}(A) \snorm{B} C_{K} + 1)} \paren{\frac{(|\lambda_1| + \varepsilon_1)^2 (1+\varepsilon_3 |\lambda_k|)}{|\lambda_k|}}^{-(\tau-1)}, \\
      &\qquad\qquad\qquad \frac{1}{4 \zeta_{\varepsilon_2}(M_2) \snorm{B} C_{K}} \paren{\frac{(|\lambda_1| + \varepsilon_1) (|\lambda_{k+1}| + \varepsilon_2) (1+\varepsilon_3 |\lambda_k|)}{|\lambda_k|}}^{-(\tau-1)}, \\
      &\qquad\qquad\qquad \left. \frac{(|\lambda_1| + \varepsilon_1)^2}{\zeta_{\varepsilon_1}(A) \snorm{B} C_{K}} \paren{\frac{(|\lambda_1| + \varepsilon_1)^2 (1+\varepsilon_3 |\lambda_k|)}{|\lambda_k|}}^{-\tau}
    \right\rbrace, \\
  \end{align*}
  which can be simplified to ($C_{\delta}$ is a constant collecting minor factors)
  \begin{equation}\label{eq:constraint-delta:merge}
    \delta < \frac{C_{\delta} \alpha c}{\sqrt{k} \bar{\zeta}^3 (\snorm{A}+\snorm{B})} |\lambda_1|^{-2\tau}
    = O(|\lambda_1|^{-2\tau}).
  \end{equation}
  Finally, we select $\omega$ such that (see (\ref{eq:constraint-omega:1}), and note that $C_{\gamma} = O(\gamma) = O(k^{-\ell})$)
  $$\omega > \max\braced{
      \frac{\log \frac{2}{C_{\gamma}}}{\log \frac{|\lambda_k|}{(1 + \varepsilon_3 |\lambda_k|)(|\lambda_{k+1}|+\varepsilon_2)}},
      \frac{\log \frac{2 C_{\varDelta}}{C_{\gamma} \delta} + \tau \log (|\lambda_1| + \varepsilon_1)}{\log \frac{|\lambda_k|}{(1 + \varepsilon_3 |\lambda_k|)(|\lambda_{k+1}|+\varepsilon_2)}}
    },$$
  which can be reorganized as
  \begin{equation}\label{eq:constraint-omega:merge}
    \omega > \frac{\log \frac{1}{C_{\gamma}} + \log \frac{\sqrt{\xi}}{1-\xi} + 2 \log \bar{\zeta} + \log \frac{\snorm{A}}{|\lambda_1|-|\lambda_{k+1}|} + \log \frac{1}{\delta} + C_{\omega}}{\log \frac{|\lambda_k|}{|\lambda_{k+1}|}}
    = O(\ell \log k).
  \end{equation}
  Note that here $\varepsilon_1, \varepsilon_2, \varepsilon_3$ are taken to be small enough, so that
  \begin{equation}\label{eq:constraint-epsilons}
    |\lambda_{k+1}| + \varepsilon_2 < 1,\quad
    |\lambda_1| + \varepsilon_1)^2 (|\lambda_{k+1}| + \varepsilon_2) < \frac{|\lambda_k|}{1 + \varepsilon_3 |\lambda_k|},\quad
    \varepsilon_3 |\lambda_k| < 1.
  \end{equation}
  Also, the probability of sampling an admissible $x_0$ is $1-\theta(\gamma) = 1 - O(k^{-\ell})$ by the union bound. This completes the proof.
\end{proof}
  \section{An Illustrative Example with Additive Noise}\label{sec:appendix-experiment}

Finally, we include an illustrative experiment that shows the performance of our \algname\ algorithm.

\textbf{Settings.} We evaluate the algorithm in LTI systems with additive noise
\begin{equation*}
  x_{t+1} = A x_t + B u_t + w_t,\quad
  \textrm{where}~w_t \mathop{\sim}^{\textrm{i.i.d.}{}} \mathcal{N}(0, \sigma_w^2 I).
\end{equation*}
Here $\sigma_w$ characterizes the variance (and thus the magnitude) of the noise. The dynamical matrices are randomly generated: $A$ is generated based on its eigen-decomposition $A = V \varLambda V^{-1}$, where the eigenvalues $\varLambda = \diag(\lambda_1, \cdots, \lambda_n)$ are randomly generated by selecting $\lambda_{1:k} \sim \mathcal{U}(1, \lambda_{\max})$ and $\lambda_{k+1:n} \sim \frac{|\lambda_k|}{|\lambda_1|^2} \cdot \mathcal{U}(-1,1)$ (to ensure $|\lambda_1|^2 |\lambda_{k+1}| < |\lambda_k|$), and the eigenvectors $V = [v_1, \cdots, v_n]$ are generated by random perturbation to a random orthogonal matrix (to avoid tiny $\xi$); meanwhile, $B$ is generated by random sampling i.i.d. entries from $\mathcal{U}(0,1)$. For comparability and reproducibility, throughout the experiment we set $k = 3$ and use $0$ as the initial random seed.

To compare the performance in different settings, 30 data points are collected for each pair of $\sigma_w$ and $n$. It is observed that our algorithm might cause numerical instability issues (e.g., $\cond(D^{\top} D)$ could be large), so we simply ignore such cases and repeat until 30 data points are collected. The parameters of the algorithm are determined in an adaptive way that minimizes the number of running steps: we search for the minimum $t_0$ that yields estimation error smaller than $\delta$, search for the minimum $\tau$ such that $K = B_{\tau}^{-1} M_1^{\tau} P_1^{\top}$ stabilizes the system, and the $\omega$ heat-up steps in Stage 3 could be ended earlier if we already observe $\snorm{\hat{P}_1^{\top} x} / \snorm{x}$ larger than a certain threshold.

Our experimental results are presented in Figure \ref{fig:experiment} below.

\begin{figure}[htbp]
  \centering
  \begin{subfigure}{0.49\linewidth}
    \includegraphics[width=0.93\linewidth]{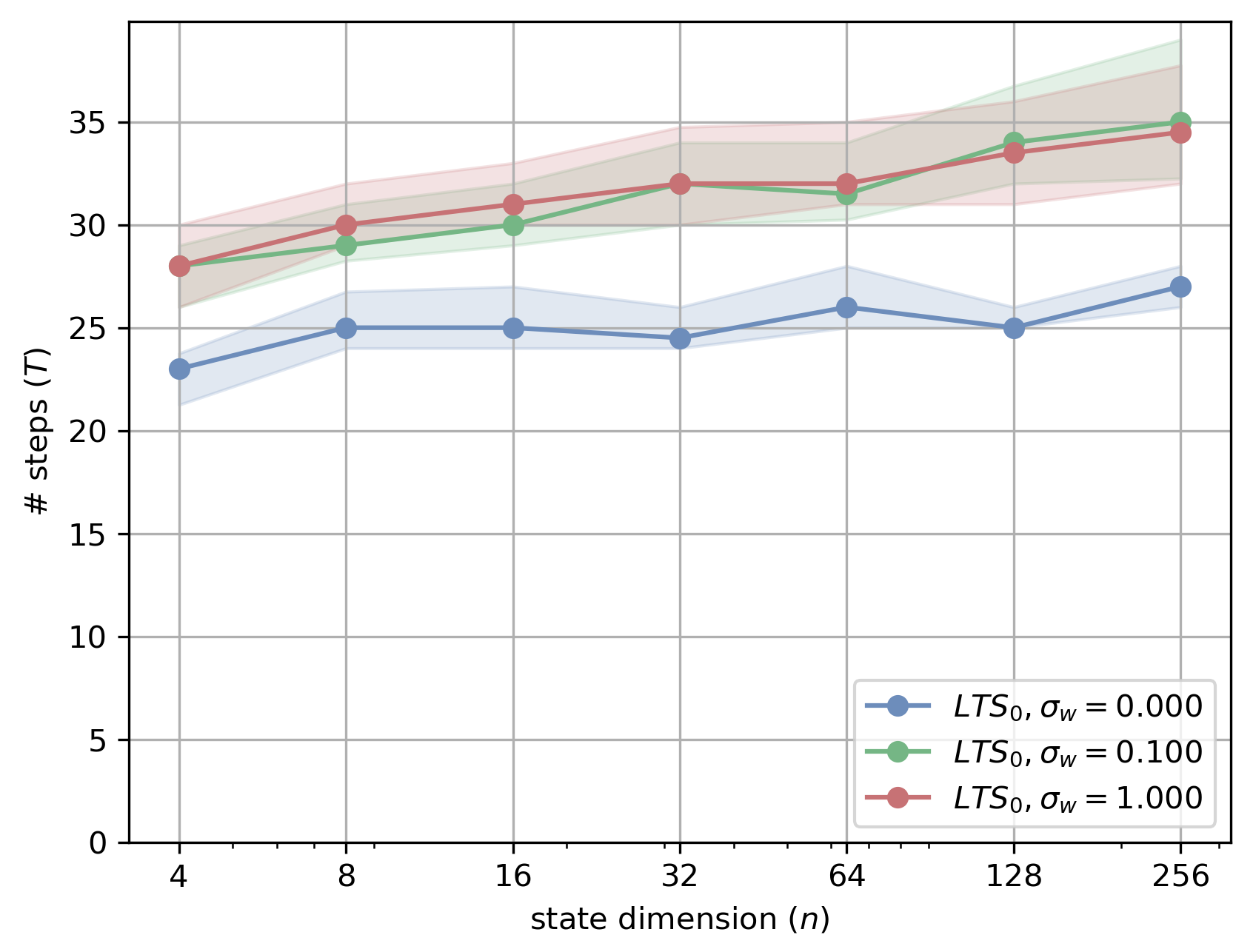}
    \caption{Running steps of \algname.}\label{fig:experiment-a}
  \end{subfigure}
  \begin{subfigure}{0.49\linewidth}
    \includegraphics[width=0.96\linewidth]{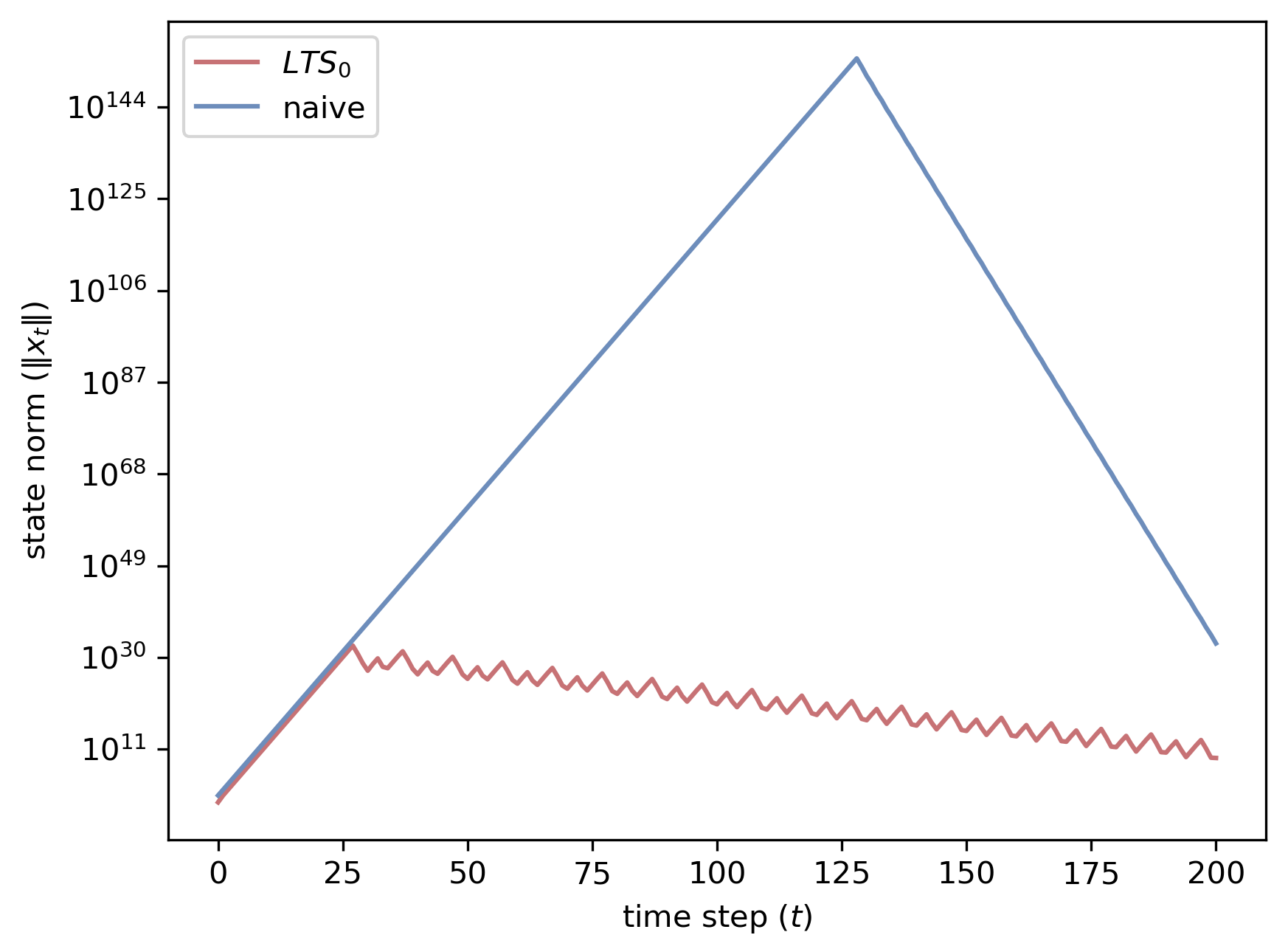}
    \caption{State norms along one trajectory.}\label{fig:experiment-b}
  \end{subfigure}
  \caption{Experimental results. \textit{In (a), the line shows the median of running steps, and the shadow marks the range between upper and lower quartiles (the horizontal axis is in log scale). In (b), the trajectories of our algorithm and the naive approach are compared in a randomly-generated system with $n = 128$ and $\sigma_w = 0$ (the vertical axis is in log scale).}}\label{fig:experiment}
\end{figure}

\textbf{Performance under different $\bm{n}$ and $\bm{\sigma_w}$.} Figure \ref{fig:experiment-a} shows the number of running steps of \algname\ that is needed to learn a stabilizing controller. It is evident that the number of running steps grow almost linearly with regard to $\log n$, which is in accordance with Theorem \ref{thm:main-theorem}.

As for the effect of noise, it is observed that the algorithm needs more steps in systems with noise than in those without noise; nevertheless, the magnitude of noise does not have much influence on the number of running steps. This is also reasonable since the increase is mainly attributed to $t_0$ --- it takes more initial steps to push the state close enough to $E_{\mathrm{u}}$, such that the estimation error of $P_1$ drops to acceptable level; however, as the $E_{\mathrm{u}}$-component grows exponentially fast over time while $w_t$ is i.i.d., the magnitude of noise only plays a minor role in the increase. Noise becomes negligible in later stages due to the disproportionate magnitudes of states and noise.

\textbf{Analysis of comparison of trajectories.} In Figure \ref{fig:experiment-b} we study an exemplary trajectory of our \algname\ algorithm, and compare it against that of the naive approach, which first identifies the system and then designs a controller to nullify the unstable eigenvalues by standard pole-placement method. It is evident that our algorithm needs significantly fewer steps, and thus induces far smaller state norms, to learn a controller that effectively stabilizes the system. It is also observed that our controller decreases state norm in a zig-zag manner, which is due to the $\tau$-hop design our algorithm adopts. Nevertheless, a potential drawback of our controller design is that the spectral radius of the controlled system is larger (since we cannot precisely nullify all unstable eigenvalues), resulting in a slower stabilizing rate than the naive approach (compare the decreasing parts of the curves).
\end{document}